\documentclass[preprint,12pt]{elsarticle}

\usepackage{amsthm}
\usepackage{epsfig}
\usepackage{amssymb}
\usepackage{amsmath}
\usepackage{float}
\usepackage{hyperref}
\usepackage{algorithm}
\usepackage{algorithmic}
\usepackage{graphicx}
\usepackage{caption}
\usepackage{subcaption}
\usepackage{mathrsfs}
\usepackage[normalem]{ulem}
\usepackage{color}

\def\N{{\mathbb{N}}}

\def\E{{\mathbb{E}}}

\newtheorem{theorem}{Theorem}
\newtheorem{lemma}{Lemma}
\newtheorem{open question}{Open Question}

\newtheorem{corollary}{Corollary}

\theoremstyle{definition}

\journal{ArXiv}

\begin{document}

\begin{frontmatter}



\title{Convex pentagons that admit $i$-block transitive tilings}


\author[1]{Casey Mann}
\author[2]{Jennifer McLoud-Mann}
\author[3]{David Von Derau}

\address[1]{University of Washington Bothell}
\address[2]{University of Washington Bothell}
\address[3]{University of Washington Bothell}
\begin{abstract}

The problem of classifying the convex pentagons that admit tilings of the plane is a long-standing unsolved problem. Previous to this article, there were 14 known distinct kinds of convex pentagons that admit tilings of the plane. Five of these types admit tile-transitive tilings (i.e. there is a single transitivity class with respect to the symmetry group of the tiling). The remaining 9 types do not admit tile-transitive tilings, but do admit either 2-block transitive tilings or 3-block transitive tilings; these are tilings comprised of clusters of 2 or 3 pentagons such that these clusters form tile-2-transitive or tile-3-transitive tilings. In this article, we present some combinatorial results concerning pentagons that admit $i$-block transitive tilings for $i \in \mathbb{N}$. These results form the basis for an automated approach to finding all pentagons that admit $i$-block transitive tilings for each $i \in \mathbb{N}$. We will present the methods of this algorithm and the results of the computer searches so far, which includes a complete classification of all pentagons admitting 1-, 2-, and 3-block transitive tilings, among which is a new 15th type of convex pentagon that admits a tile-3-transitive tiling.

\end{abstract}

\begin{keyword}


tiling \sep pentagon

\end{keyword}

\end{frontmatter}


\section{Preliminaries}

A \emph{plane tiling $\mathscr{T}$} is a countable family of closed topological disks $\mathscr{T}=\lbrace T_1,T_2,...\rbrace$ that cover the Euclidean plane $\E^2$ without gaps or overlaps; that is, $\mathscr{T}$ satisfies
\begin{enumerate}
\item $\displaystyle \bigcup_{i \in \N} T_i = \E^2$, and
\item $\text{int}(T_i) \cap \text{int}(T_j) = \emptyset$ when $i \neq j$.
\end{enumerate}
\noindent The $T_i$ are called the \emph{tiles} of $\mathscr{T}$.  If the tiles of $\mathscr{T}$ are all congruent to a single tile $T$, then $\mathscr{T}$ is \emph{monohedral} with \emph{prototile} $T$ and we say that the prototile $T$ \emph{admits} the tiling $\mathscr{T}$. The intersection of any two distinct tiles of $\mathscr{T}$ can be a set of isolated arcs and points. These arcs are called the \emph{edges} of $\mathscr{T}$, and the isolated points, along with the endpoints of the edges, are called the \emph{vertices} of $\mathscr{T}$. In this paper, only tilings whose tiles are convex polygons are considered. To distinguish between features of the tiling and features of the polygons comprising the tiling, the straight segments forming the boundary of a polygon will be called its \emph{sides} and the endpoints of these straight segments will be called its \emph{corners}. If the corners and sides of the polygons in a tiling coincide with the vertices and edges of the tiling, then the tiling is said to be \emph{edge-to-edge}.

A \emph{symmetry} of $\mathscr{T}$ is an isometry of $\E^2$ that maps the tiles of $\mathscr{T}$ onto tiles of $\mathscr{T}$, and the \emph{symmetry group} $S(\mathscr{T})$ of $\mathscr{T}$ is the collection of such symmetries. If $S(\mathscr{T})$ contains two nonparallel translations, $\mathscr{T}$ is \emph{periodic}. Two tiles $T_1,T_2 \in \mathscr{T}$ are said to be \emph{equivalent} if there is an isometry $\sigma \in \mathscr{T}$ such that $\sigma(T_1) = T_2$. If all tiles of $\mathscr{T}$ are equivalent, $\mathscr{T}$ is said to be \emph{tile-transitive} (or \emph{isohedral}). Similarly, if there are exactly $k$ distinct transitivity classes of tiles of $\mathscr{T}$ with respect to $S(\mathscr{T})$, then $\mathscr{T}$ is \emph{tile-$k$-transitive}. The tile-transitive tilings of the plane have been classified \cite{GS2}, and this classification will be central to the methodology presented in this article.

The tiles of $\mathscr{T}$ are \emph{uniformly bounded} if there exist parameters $u,U>0$ such that every tile of $\mathscr{T}$ contains a disk of radius $u$ and is contained in a disk of radius $U$. A tiling $\mathscr{T}$ is \emph{normal} if three conditions hold:
\begin{enumerate}
\item Each tile of $\mathscr{T}$ is a topological disk,
\item The intersection of any two tiles of $\mathscr{T}$ is a connected set, and
\item The tiles of $\mathscr{T}$ are uniformly bounded.
\end{enumerate}

\noindent The \emph{patch} of $\mathscr{T}$ generated by the disk $D(r,P)$ of radius $r$ centered at point $P$ is the set of tiles $\mathscr{A}(r,P) \subset \mathscr{T}$ that meet $D(r,P)$, along with any additional tiles in $\mathscr{T}$ required to make the union of the tiles in $\mathscr{A}(r,P)$ a closed topological disk. The numbers of tiles, edges, and vertices of $\mathscr{T}$ contained in $\mathscr{A}(r,P)$ will be denoted by $t(r,P)$, $e(r,P)$, and $v(r,P)$, respectively. The following is a fundamental result concerning normal tilings.

\begin{theorem}[Normality Lemma \cite{GS1}] Let $\mathscr{T}$ be a normal tiling. Then for any real number $x > 0$, \[\lim_{r \rightarrow \infty} \frac{t(r+x,P)}{t(r,P)} = 1.\]\end{theorem}

A normal tiling $\mathscr{T}$ is \emph{balanced} if the following limits exist. \begin{equation} v(\mathscr{T}) = \lim_{r \rightarrow \infty}\frac{v(r,P)}{t(r,P)} \text{ \hspace{.2in} and \hspace{.2in} } e(\mathscr{T}) = \lim_{r \rightarrow \infty}\frac{e(r,P)}{t(r,P)}\label{eqn:balanced}\end{equation} Balanced tilings have the following nice property.

\begin{theorem}[Euler's Theorem for Tilings \cite{GS1}] For any normal tiling $\mathscr{T}$, if either of the limits $v(\mathscr{T})$ or $e(\mathscr{T})$ exists (and is finite), then so does the other. Moreover, if either of the limits $v(\mathscr{T})$ or $e(\mathscr{T})$ exists, $\mathscr{T}$ is balanced and \begin{equation} v(\mathscr{T}) = e(\mathscr{T}) - 1.\label{eqn:Euler}\end{equation}\end{theorem}

\subsection{Monohedral Tilings by Convex Pentagons}

This article is concerned with monohedral tilings of the plane in which the prototile is a convex pentagon.  It is known that all triangles and quadrilaterals (convex or not) tile the plane. It is also known that there are exactly 3 classes of convex hexagons that tile the plane \cite{HK}. Figure \ref{fig:hex} shows how the convex hexagons that admit tilings of the plane are classified in terms of relationships among their angles and sides. 

\begin{figure}[H]
\centering
\begin{subfigure}[H]{1\textwidth} 
\centering
\includegraphics[scale=.8]{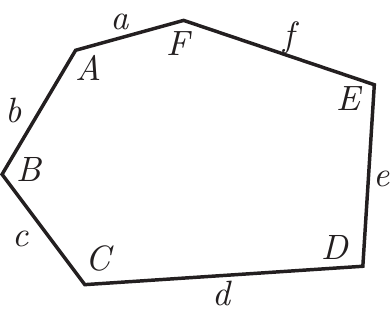}  
\caption{A labeled hexagon}\label{fig:hex-label}
\end{subfigure}
\begin{subfigure}[H]{.3\textwidth}
\centering
\includegraphics[scale=.8]{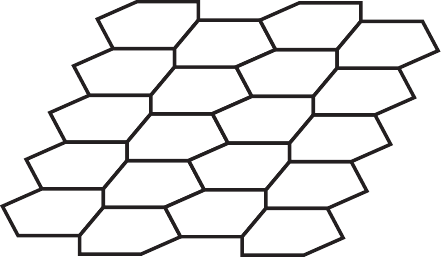}  
\caption{$A+B+C = 2\pi$;\\ $a = d$}
\end{subfigure}
\begin{subfigure}[H]{.3\textwidth}
\centering
\includegraphics[scale=.8]{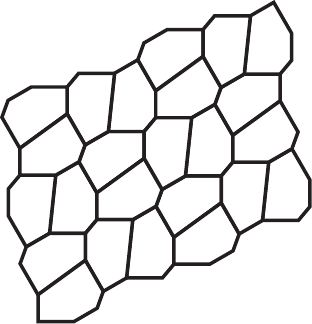}  
\caption{$A+B+D = 2\pi$; \\$a=d$; $c = e$}
\end{subfigure}
\begin{subfigure}[H]{.3\textwidth}
\centering
\includegraphics[scale=.8]{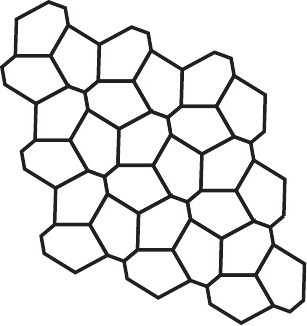} 
\caption{$A = C = E = 2\pi/3$; $a = b$; $c = d$; $e = f$} 
\end{subfigure}
\caption{The three classes of convex hexagons that admit tilings of the plane.}\label{fig:hex}
\end{figure}

It has also been shown that convex $n$-gons with $n \geq 7$ admit no tilings of the plane  \cite{GS1,Niv}. Previous to this article, there were 14 known distinct classes of convex pentagons that tile the plane (Figure \ref{fig:pentagons}). The labeling system for pentagons is the same as that of the hexagons in Figure \ref{fig:hex}. The first 5 types admit tile-transitive tilings of the plane; it was shown by K. Reinhardt that any convex pentagon admitting a tile-transitive tiling of the plane is one of these 5 types. Types 6-8 were discovered by Kershner \cite{Ker}, Type 9 and 11-13 were discovered by M. Rice, and Type 10 by  R. James \cite{Sch}. In \cite{Sch}, D. Schattschneider gives an interesting history (up to 1978) of the problem of classifying convex pentagons that admit tilings of the plane. Since that time, the 14th type of pentagon was discovered in 1985 by R. Stein, and large categories of pentagons have been shown to admit only tilings from among the knows 14 types; this includes equilateral pentagons (\cite{OB1,HH}) and pentagons that admit edge-to-edge tilings \cite{OB2}. In this article we will present a new type of pentagon (Type 15), as well as the results of our exhaustive computer search for convex pentagons that admit $i$-block transtive tilings for $i = 1, 2$, and $3$.

\begin{figure}[H]
\centering
\begin{subfigure}[H]{.3\textwidth} 
\centering
\includegraphics[scale=.15]{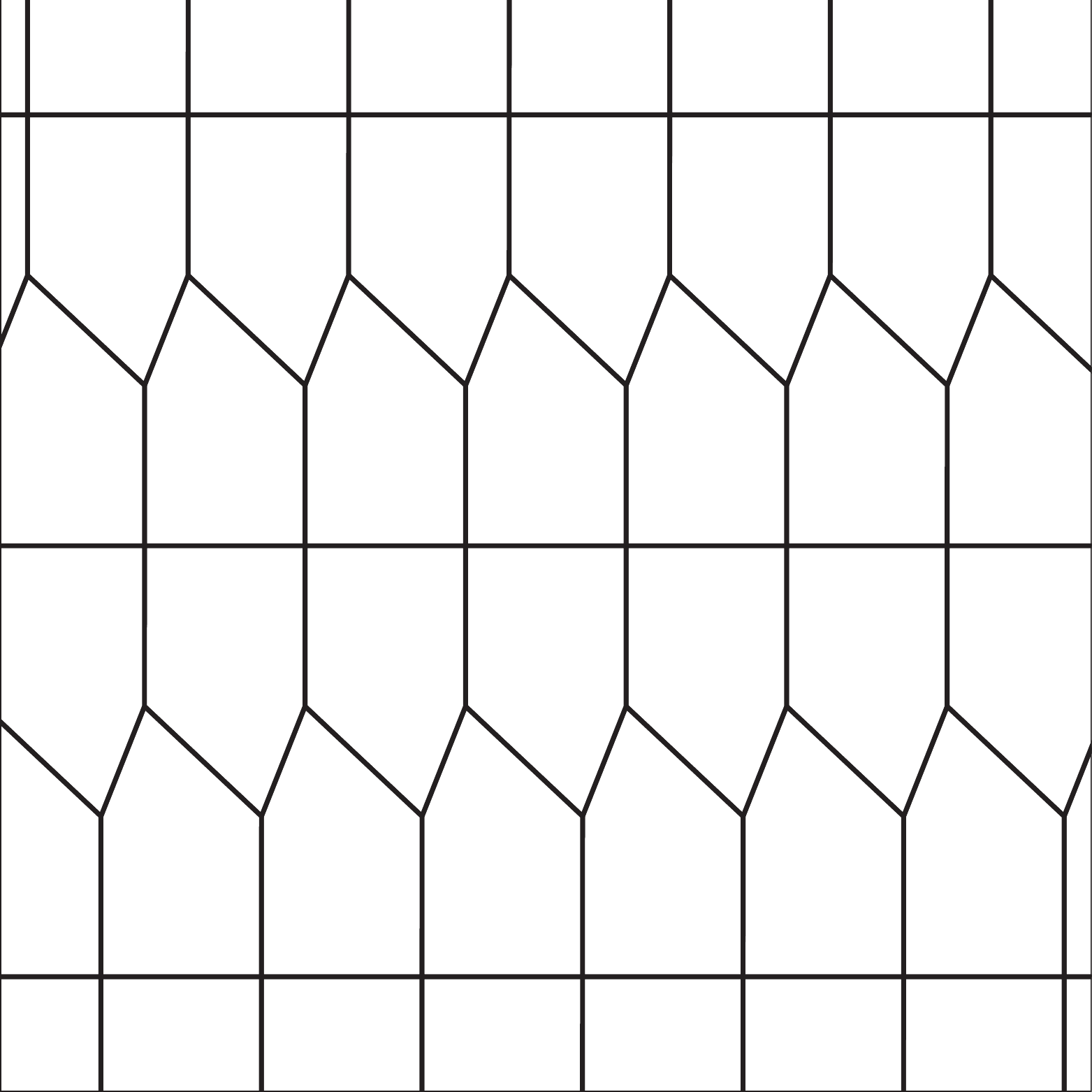}  
\caption*{\tiny Type 1\\$D+E=\pi$\\ \mbox{}}
\end{subfigure}
\begin{subfigure}[H]{.3\textwidth} 
\centering
\includegraphics[scale=.15]{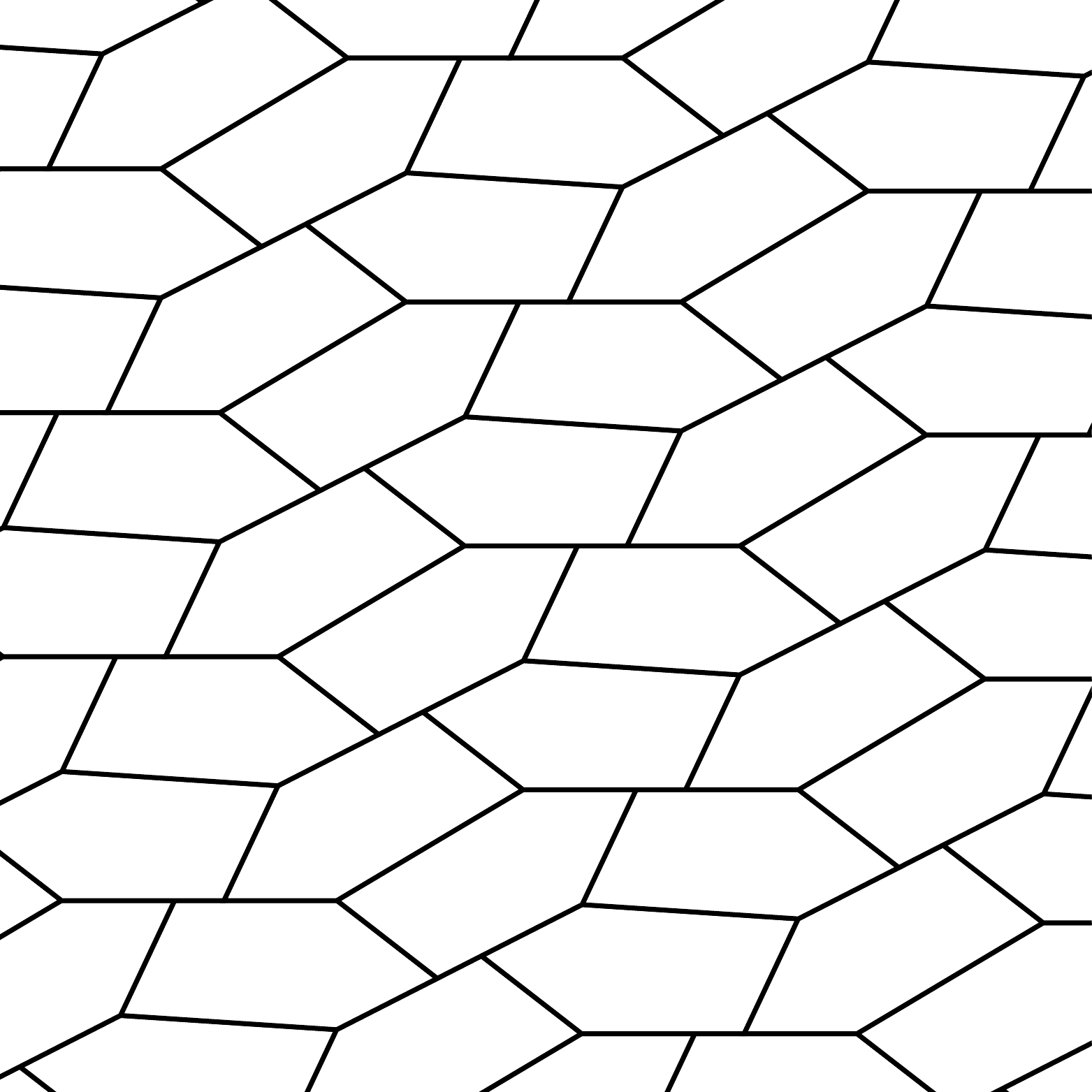}  
\caption*{\tiny Type 2\\$C+E = \pi$; $a = d$\\ \mbox{}}
\end{subfigure}
\begin{subfigure}[H]{.3\textwidth} 
\centering
\includegraphics[scale=.15]{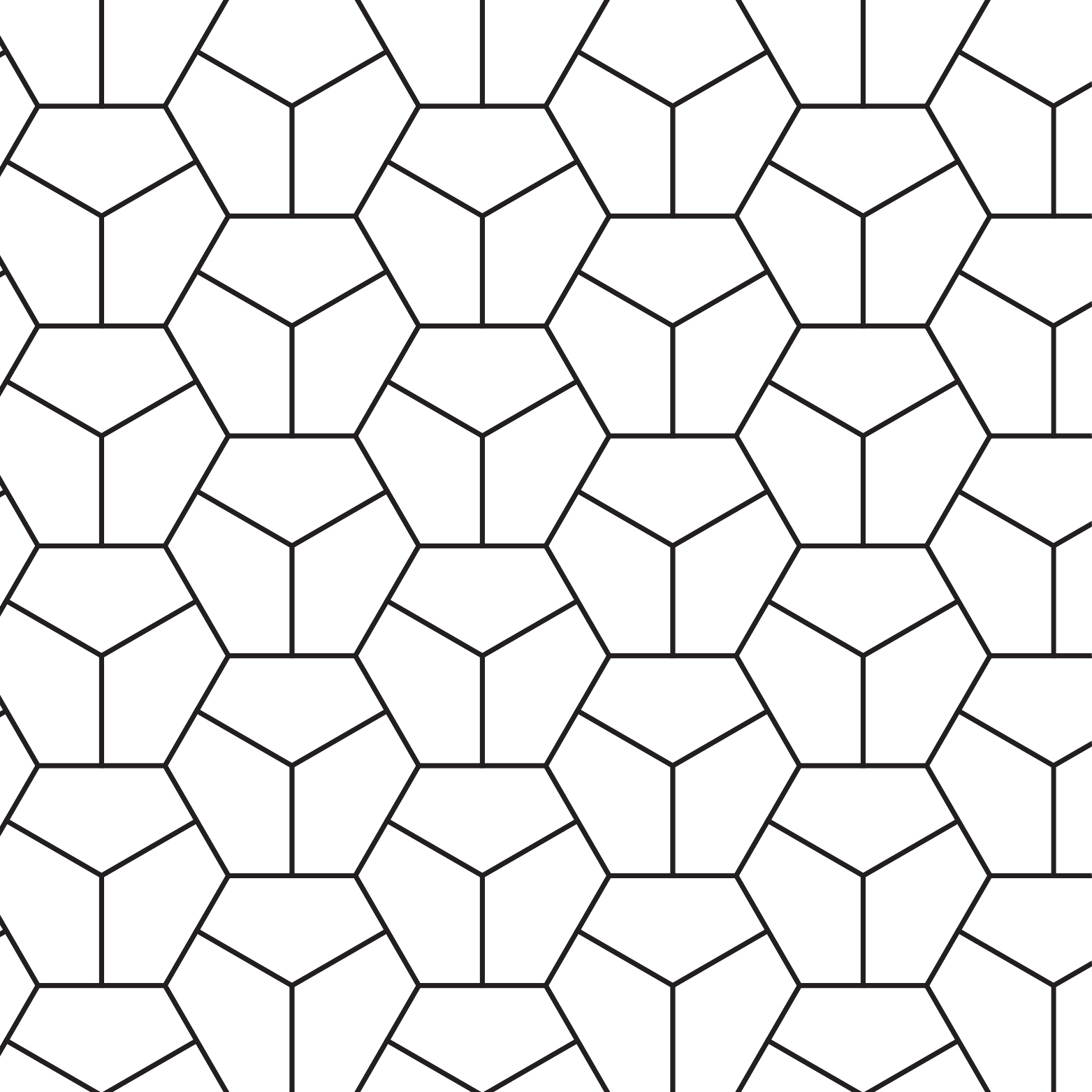}  
\caption*{\tiny Type 3\\$A=C=D=2\pi/3$;\\$a=b,d=c+e$}
\end{subfigure}
\begin{subfigure}[H]{.3\textwidth} 
\centering
\includegraphics[scale=.15]{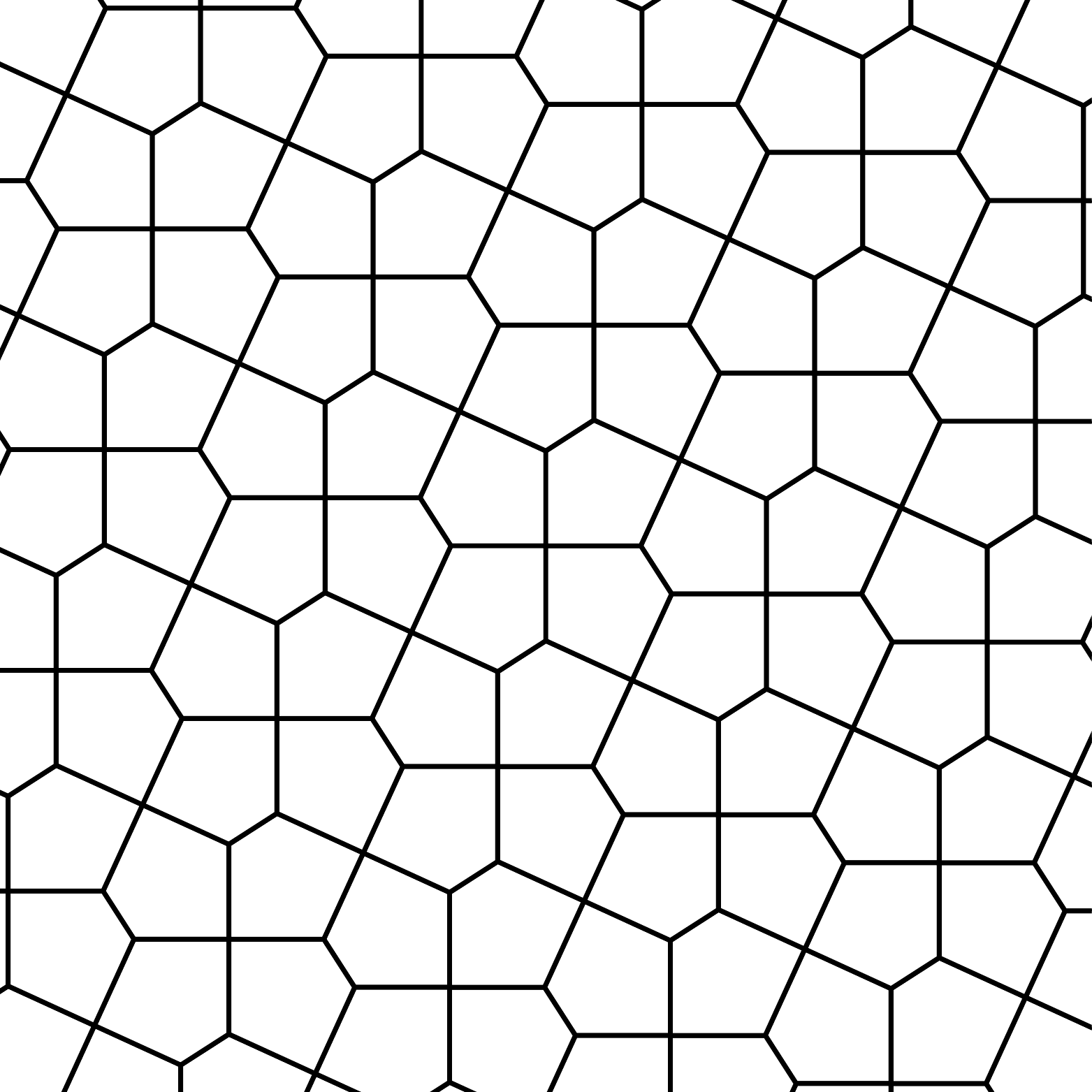}  
\caption*{\tiny Type 4\\$A=C=\pi/2$; \\$a = b, c= d$}
\end{subfigure}
\begin{subfigure}[H]{.3\textwidth} 
\centering
\includegraphics[scale=.15]{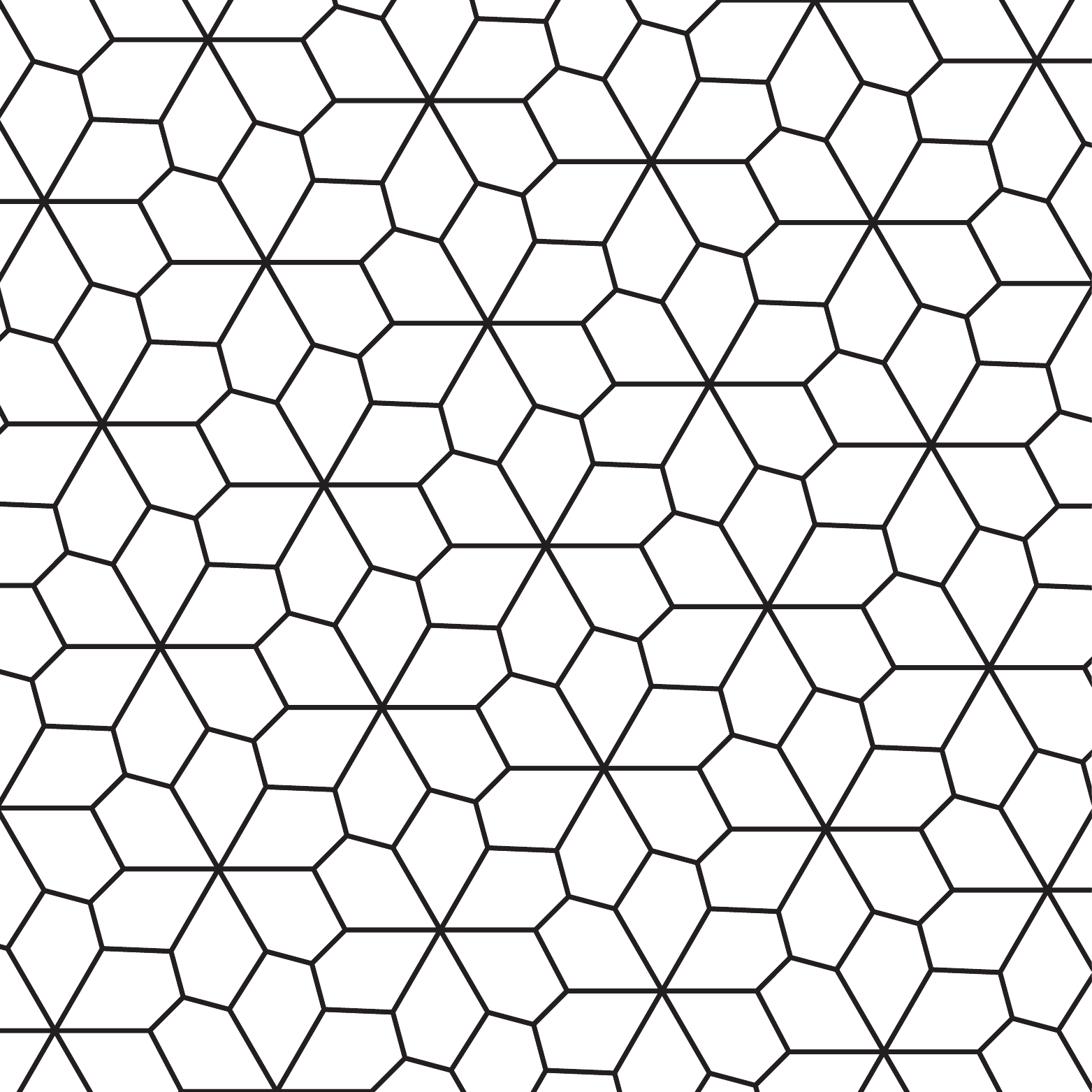}  
\caption*{\tiny Type 5\\$C = 2A = \pi/2$;\\$a= b$, $c = d$}
\end{subfigure}
\begin{subfigure}[H]{.3\textwidth} 
\centering
\includegraphics[scale=.15]{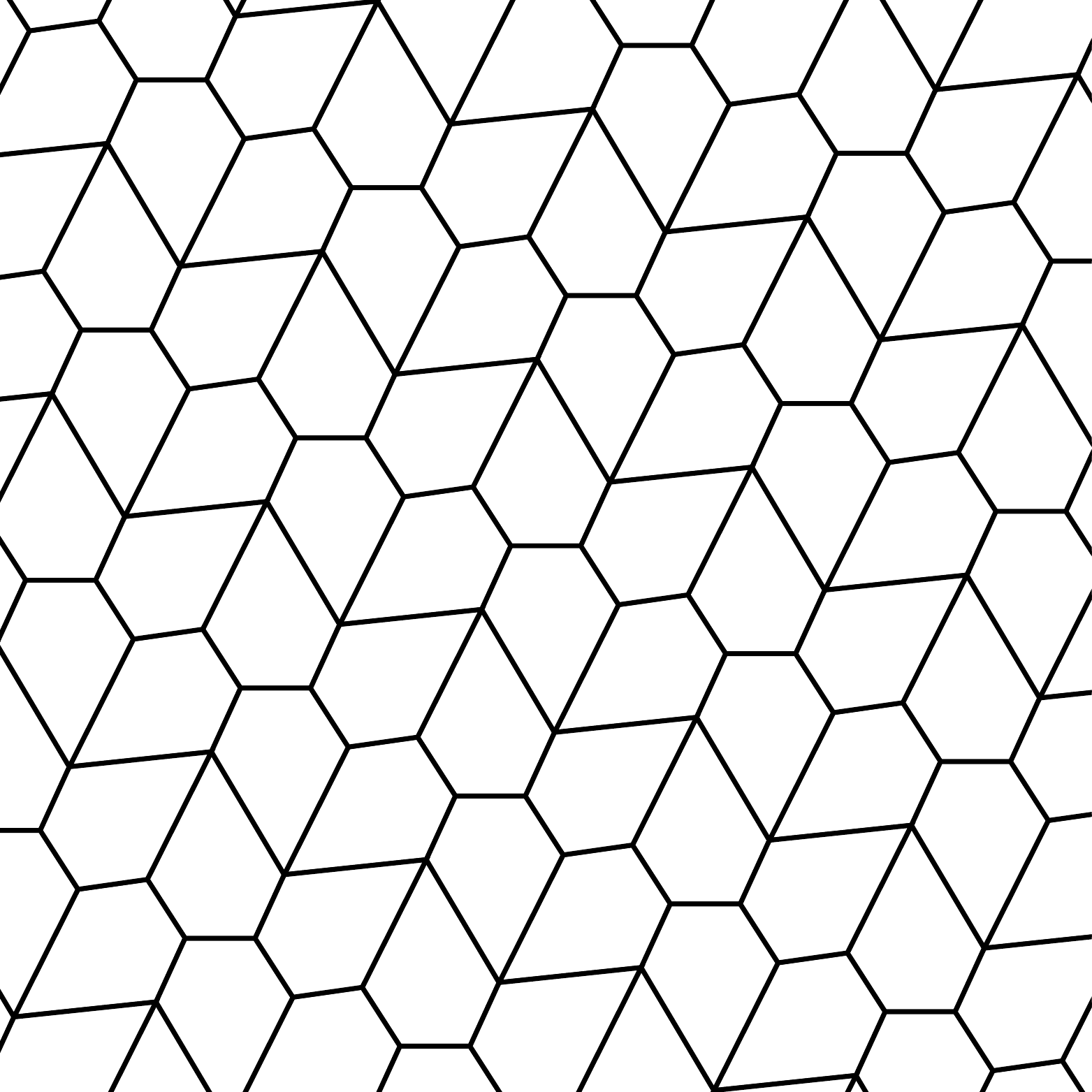}  
\caption*{\tiny Type 6\\$C+E = \pi$,$A = 2C$;\\$a=b=e,c=d$}
\end{subfigure}
\begin{subfigure}[H]{.3\textwidth} 
\centering
\includegraphics[scale=.15]{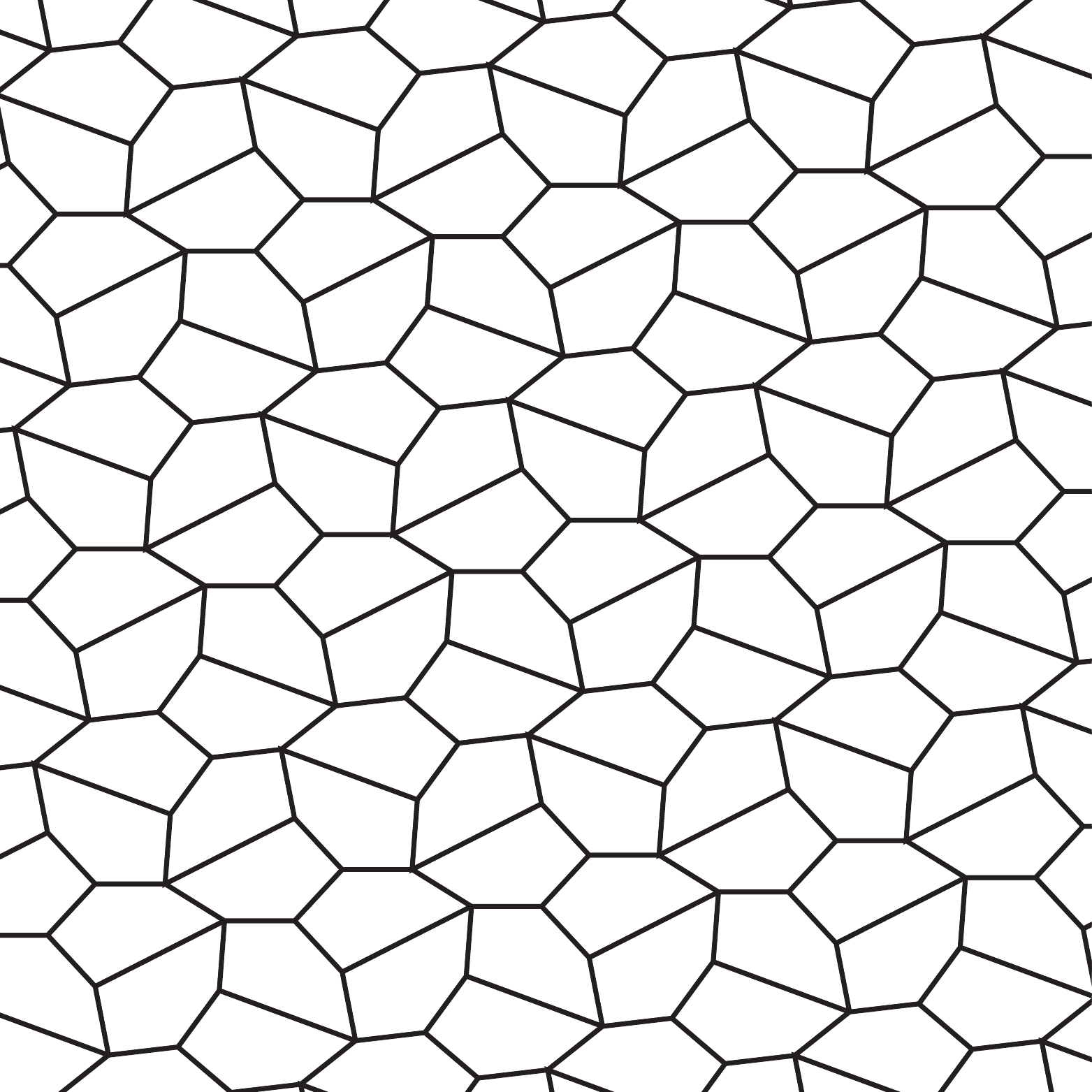}  
\caption*{\tiny Type 7\\$2B+C = 2\pi$,\\$2D+A=2\pi$;\\$a=b=c=d$}
\end{subfigure}
\begin{subfigure}[H]{.3\textwidth} 
\centering
\includegraphics[scale=.15]{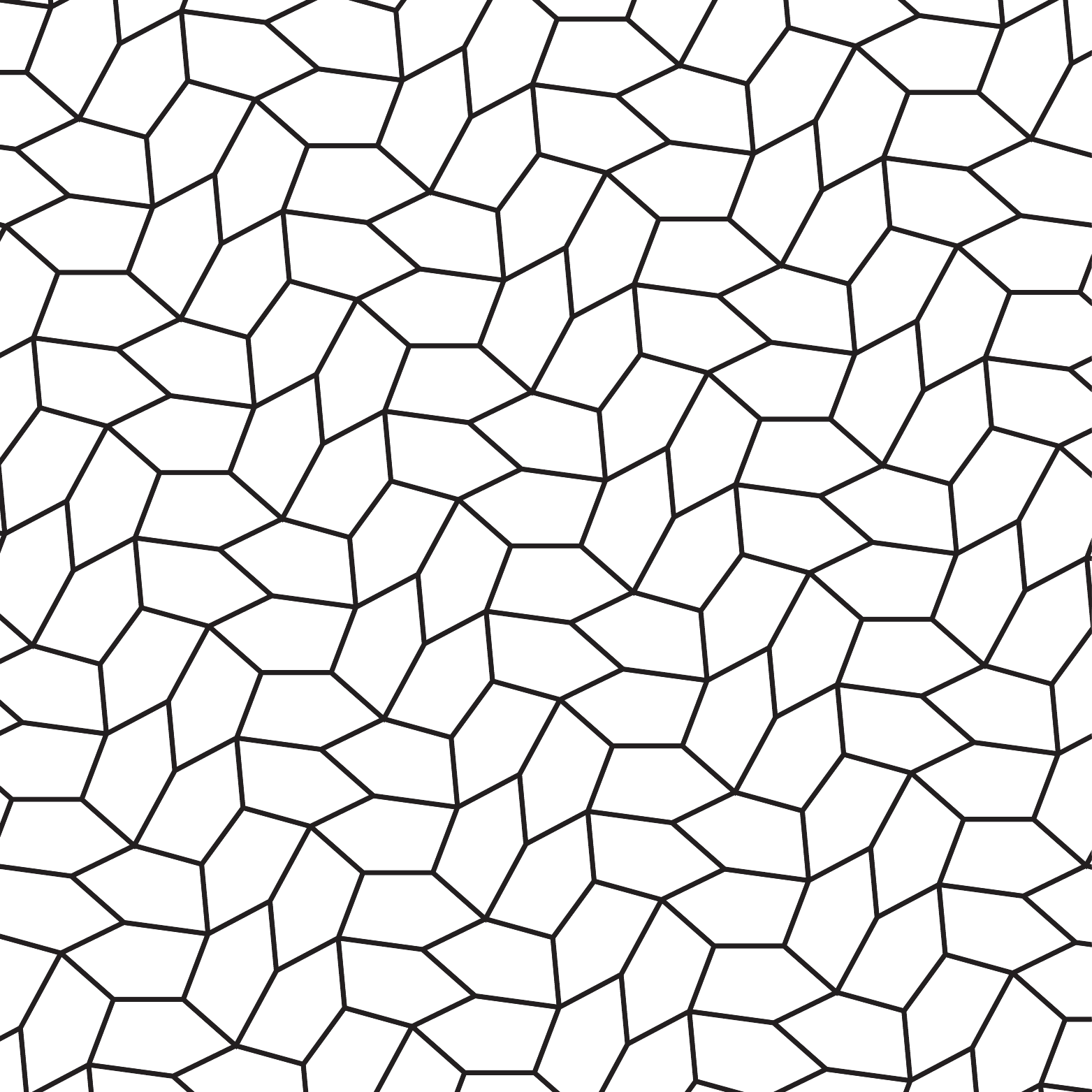}  
\caption*{\tiny Type 8\\$2A+B = 2\pi$,\\$2D+C = 2\pi$;\\$a=b=c=d$}
\end{subfigure}
\begin{subfigure}[H]{.3\textwidth} 
\centering
\includegraphics[scale=.15]{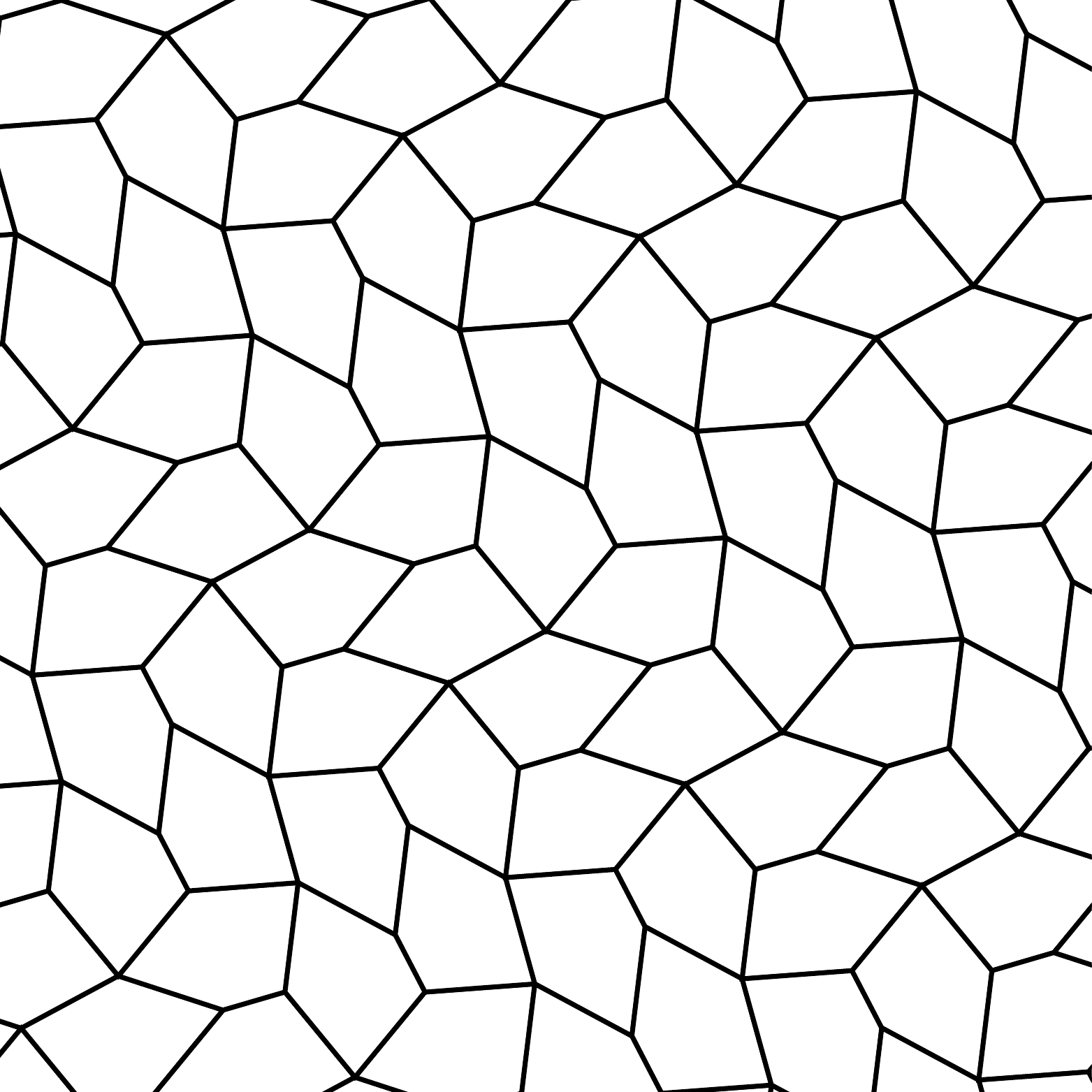}  
\caption*{\tiny Type 9\\$2E+B = 2\pi$,\\$2D+C = 2\pi$;\\$a=b=c=d$}
\end{subfigure}
\begin{subfigure}[H]{.3\textwidth} 
\centering
\includegraphics[scale=.15]{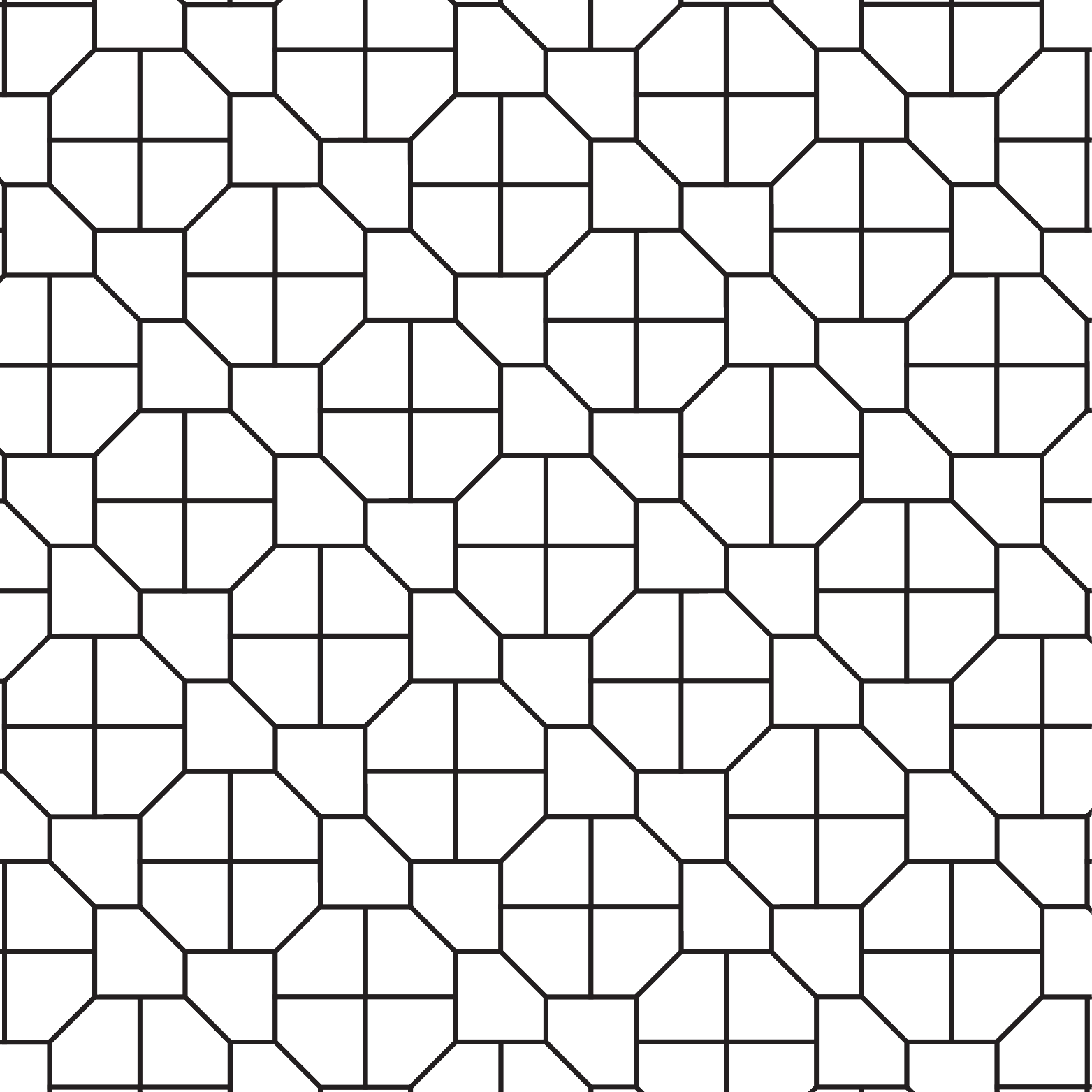}  
\caption*{\tiny Type 10\\$E = \pi/2$, $A+D = \pi$,\\$2B-D = \pi$,$2C+D=2\pi$;\\$a=e=b+d$}
\end{subfigure}
\begin{subfigure}[H]{.3\textwidth} 
\centering
\includegraphics[scale=.15]{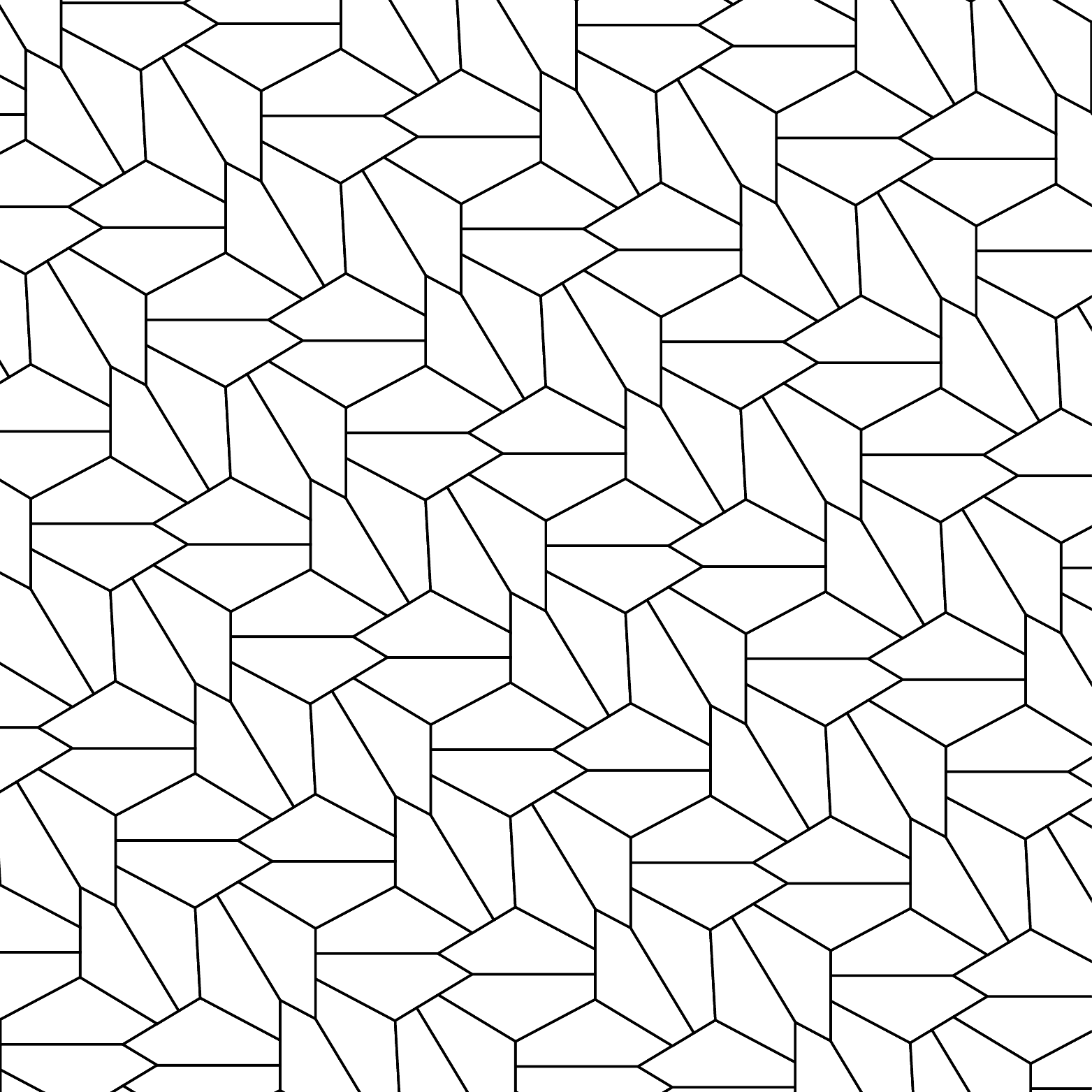}  
\caption*{\tiny Type 11\\$A=\pi/2$, $C+E = \pi$,\\$2B+C = 2\pi$;\\$d = e = 2a + c$}
\end{subfigure}
\begin{subfigure}[H]{.3\textwidth} 
\centering
\includegraphics[scale=.15]{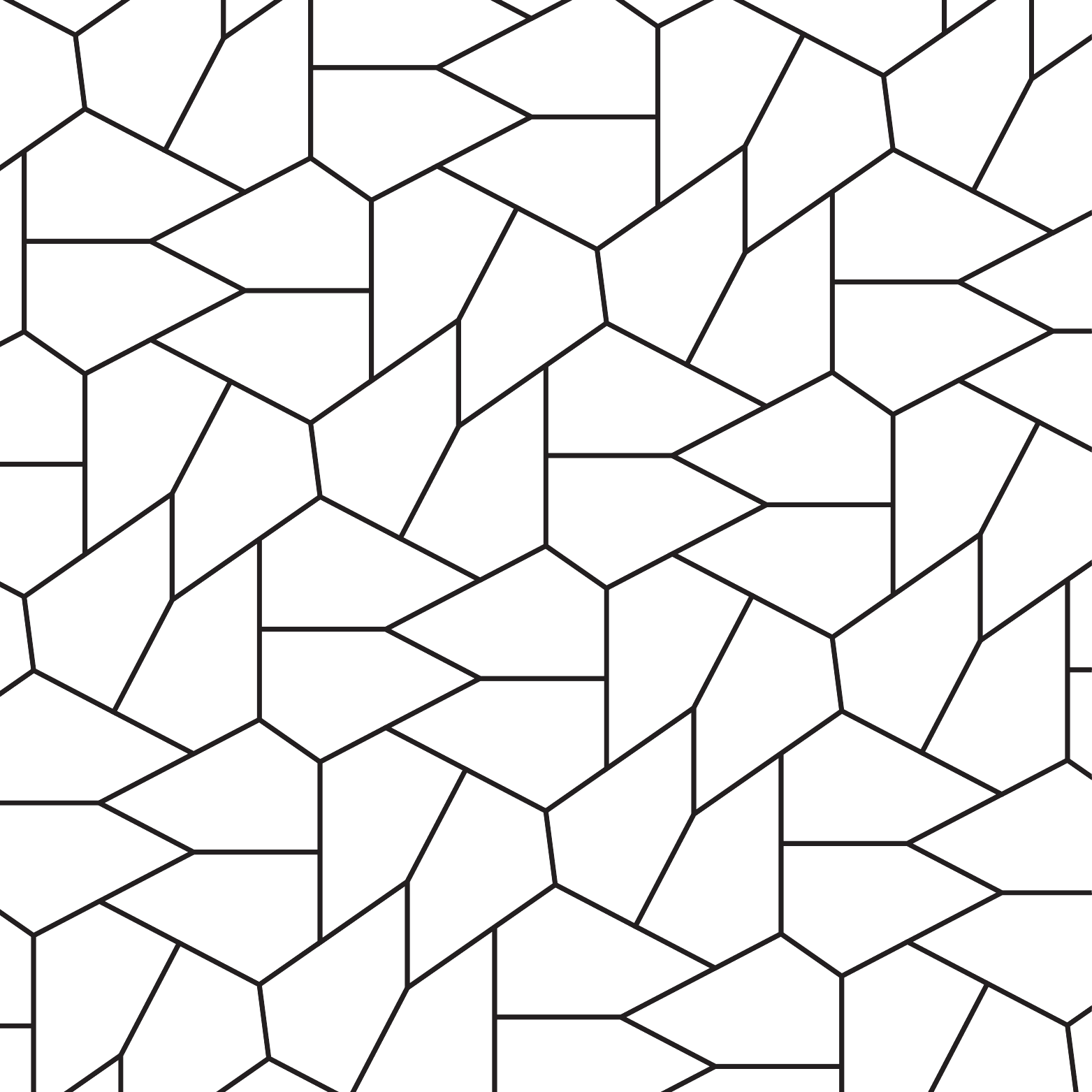}  
\caption*{\tiny Type 12\\$A = \pi/2$, $C+E=\pi$,\\$2B+C = 2\pi$;\\$2a=c+e=d$}
\end{subfigure}
\begin{subfigure}[H]{.3\textwidth} 
\centering
\includegraphics[scale=.15]{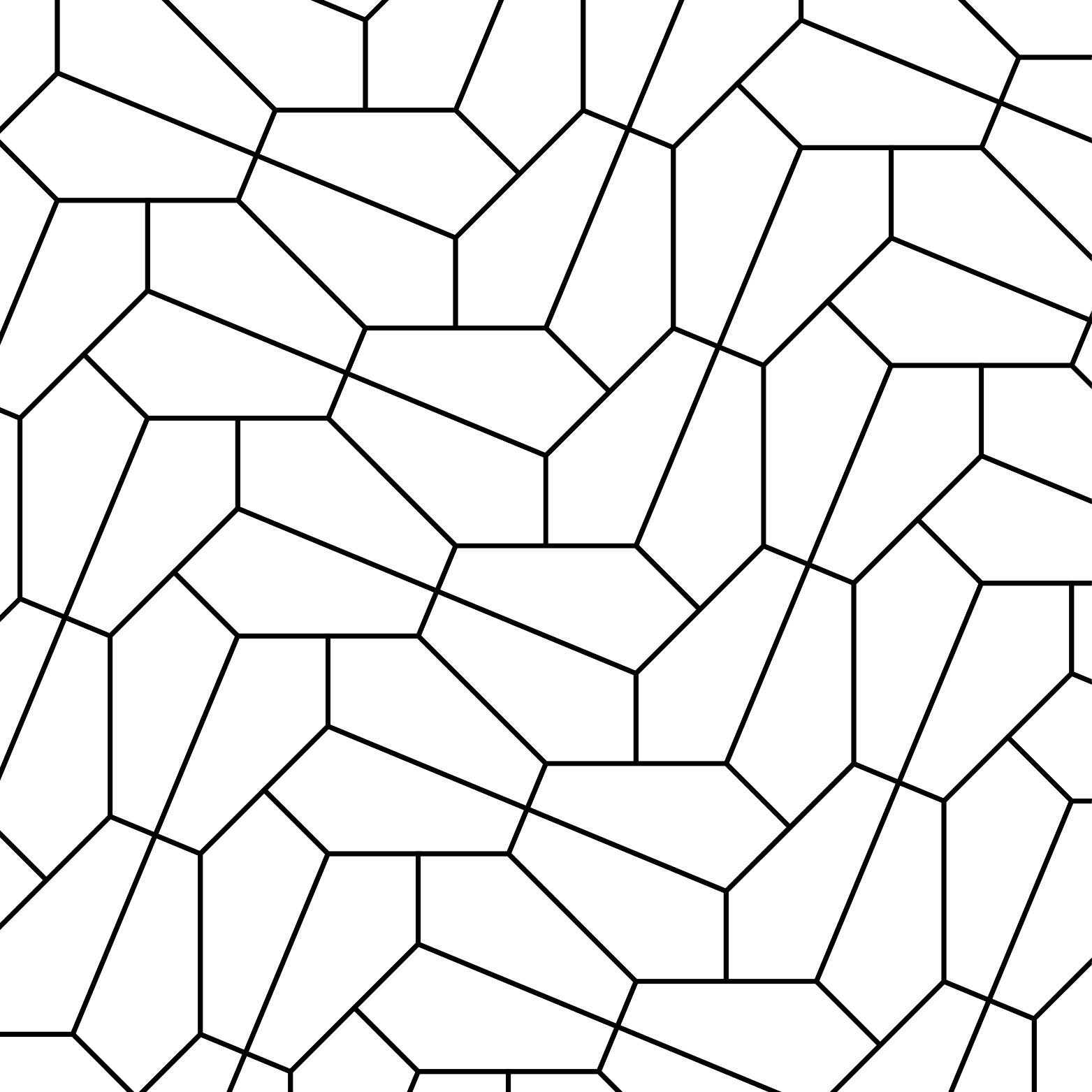}  
\caption*{\tiny Type 13\\$A=C=\pi/2$,\\$2B=2E=2\pi-D$;\\$c=d$, $2c=e$}
\end{subfigure}
\begin{subfigure}[H]{.3\textwidth} 
\centering
\includegraphics[scale=.15]{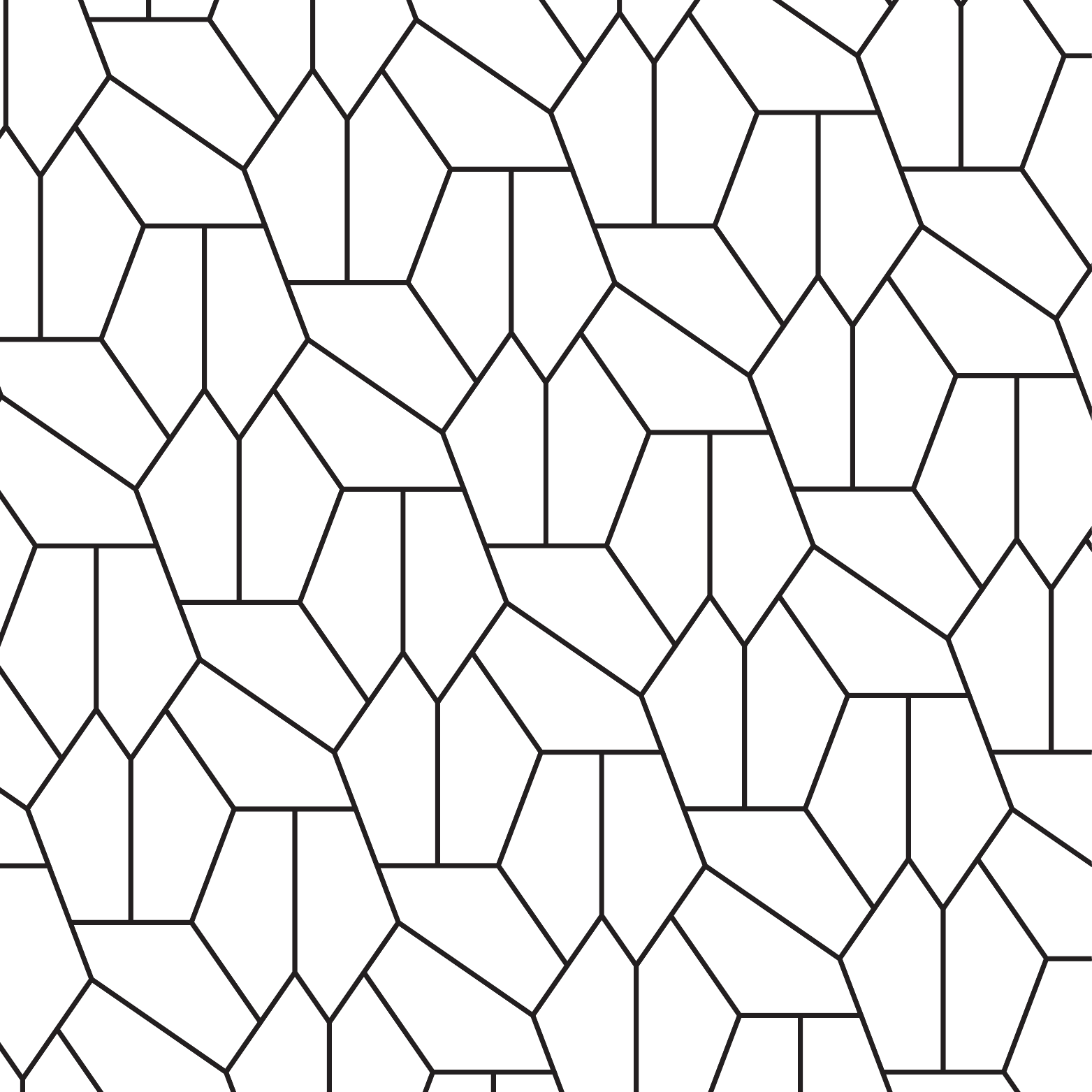}  
\caption*{\tiny Type 14\\$D=\pi/2$, $2E + A = 2\pi$,\\$A+C = \pi$;\\$b=c=2a=2d$}
\end{subfigure}\caption{Pentagon Types 1-14}\label{fig:pentagons}
\end{figure}

An \emph{$i$-block transitive tiling} $\mathscr{T}$ is a monohedral tiling by convex pentagons that contains a patch $\mathscr{B}$ conisting of $i$ pentagons such that (1) $\mathscr{T}$ consists of congruent images of $\mathscr{B}$, and (2) this corresponding tiling by copies of $\mathscr{B}$ is an isohedral tiling, and (3) $i$ is the minimum number of pentagons for which such a patch $\mathscr{B}$ exists. Such a patch $\mathscr{B}$ will be called an \emph{$i$-block}, and the corresponding isohedral tiling will be denoted by $\mathscr{I}$. 

\begin{figure}[H]
\centering
\begin{subfigure}[H]{.45\textwidth} 
\centering
\includegraphics[scale=.5]{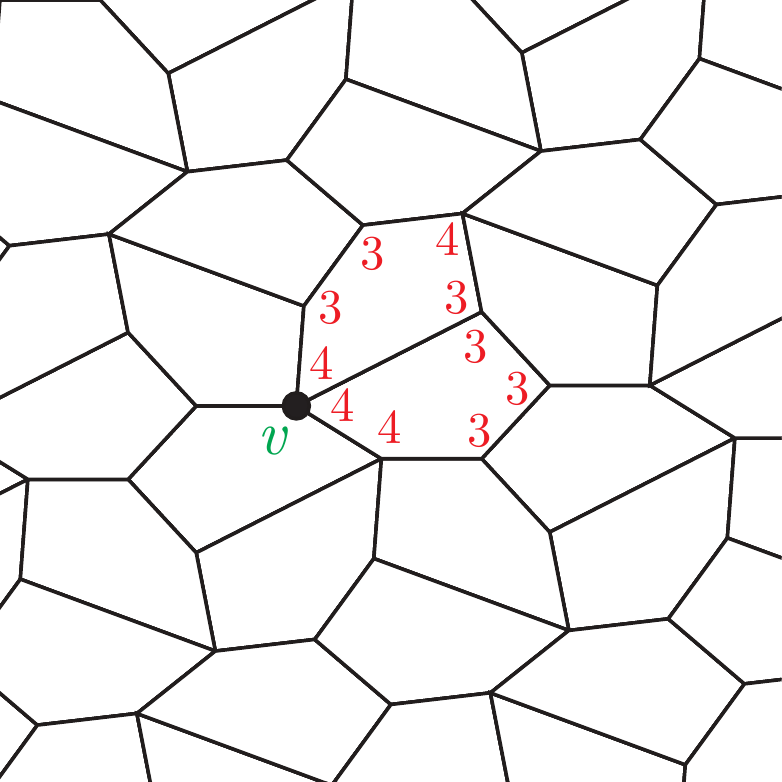}  
\caption{A type 7 tiling $\mathscr{T}$}\label{fig:type-7}
\end{subfigure}
\begin{subfigure}[H]{.45\textwidth}
\centering
\includegraphics[scale=.5]{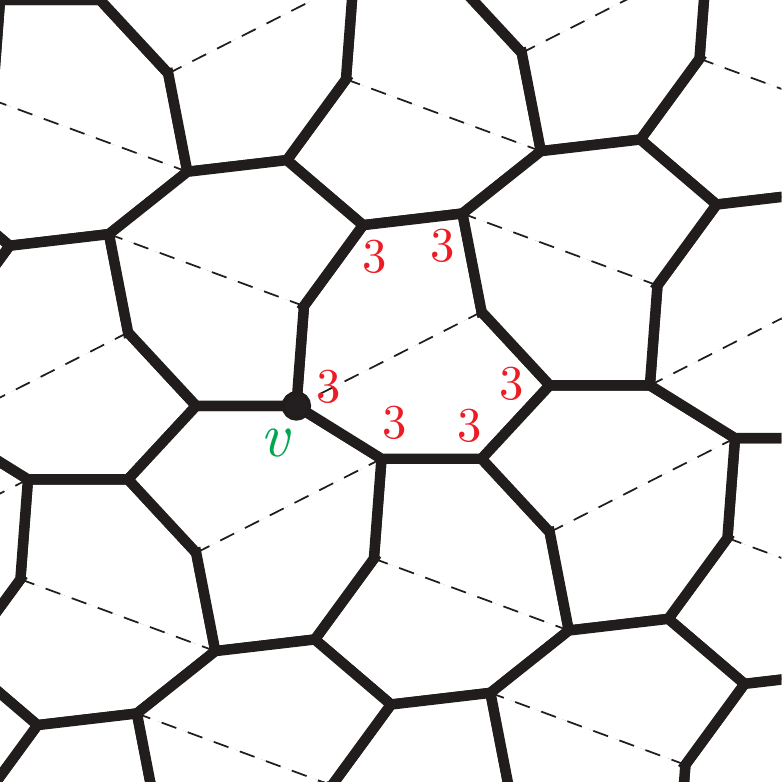}  
\caption{A corresponding tiling $\mathscr{I}$ tiling by 2-blocks.}\label{fig:2-block-tiling}
\end{subfigure}
\caption{A pentagon tiling $\mathscr{T}$ and a corresonding 2-block tiling $\mathscr{I}$}\label{fig:type-7-block}
\end{figure}  

If $v$ is a vertex of both $\mathscr{T}$ and $\mathscr{I}$, then let $V_{\mathscr{T}}(v)$ denote the valence of $v$ in $\mathscr{T}$ and let $V_{\mathscr{I}}(v)$ denote the valence of $v$ in $\mathscr{I}$. For example, for the designated vertex $v$  in Figure \ref{fig:type-7-block}, we see that $V_{\mathscr{T}}(v) = 4$, while $V_{\mathscr{I}}(v) = 3$.

Note that any periodic tiling by convex pentagons is necessarily $i$-block transitive for some $i$ (consider the pentagons comprising a fundamental region of the periodic tiling). It would be reasonable to conjecture that any unmarked convex pentagon that admits a tiling of the plane admits at least one periodic tiling; that is, it would be reasonable to conjecture that there are no aperiodic convex pentagons. If this conjecture is true, then all convex pentagons that admit tilings of the plane also admit at least one $i$-block transitive tiling. Thus, the class of pentagons being studied in this article may well encompass all possible pentagons that admit tilings of the plane.

\section{Combinatorial Results Concerning $i$-Block Transitive Tilings}

Suppose that a convex pentagon admits an $i$-block transitive tiling $\mathscr{T}$, let $\mathscr{B}$ be an $i$-block of $\mathscr{T}$, and let $P$ be any pentagon in $\mathscr{B}$. Define a \emph{node} of $P$ to be any vertex of $\mathscr{T}$ that lies on $P$. Note that the corners of $P$ are necessarily nodes, but $P$ may have nodes at other points besides its corners if the tiling is not edge-to-edge. 

\begin{theorem}  For an $i$-block transitive tiling $\mathscr{T}$ with $i$-block $\mathscr{B}$, suppose $\mathscr{B}$ has exactly $n$ nodes, counted with multplicity at nodes shared by multiple pentagons of $\mathscr{B}$, and let $\alpha_j$ denote the valence in $\mathscr{T}$ of the $j$-th node of $\mathscr{B}$. Then $\mathscr{T}$ is balanced with  \begin{equation}v(\mathscr{T})  = \frac{1}{i} \sum_{j = 1}^{n}\frac{1}{\alpha_j}\label{eqn:v}\end{equation} and \begin{equation}e(\mathscr{T}) = \frac{n}{2i}.\label{eqn:e}\end{equation}\label{thm:main} \end{theorem}

\begin{proof} All tile-transitive tilings are periodic, and from this it follows that $\mathscr{T}$ is periodic as well. Additionally, all periodic tilings are balanced \cite{GS1}, and so $\mathscr{T}$ is balanced, and so the limits $v(\mathscr{T})$ and $e(\mathscr{T})$ exist. 

To find a formula for $v(\mathscr{T})$, let $P$ be any point of the plane and let $r > 0$. In the patch $\mathscr{A}(r,P)$, \begin{equation*} v(r,P) \approx \frac{t(r,P)}{i} \sum_{j = 1}^{n}\frac{1}{\alpha_j}.\end{equation*}  The reason this estimate is not exact is due to $i$-blocks and partial $i$-blocks on the boundary of $\mathscr{A}(r,P)$ whose pentagons are not completely surrounded by other pentagons in the patch. Observe that for large $r$, \begin{equation}\frac{t(r-2iU,P)}{i} \sum_{j = 1}^{n}\frac{1}{\alpha_j} \leq v(r,P) \leq \frac{t(r+2iU,P)}{i} \sum_{j = 1}^{n}\frac{1}{\alpha_j},\label{eqn:squeeze} \end{equation} where $U$ is the circumparameter of $\mathscr{T}$. The lower bound on $v(r,P)$ holds since no $i$-block of $\mathscr{A}(r-2iU)$ meets any $i$-block on the boundary of $\mathscr{A}(r,P)$, and similarly the upper bound follows from the fact that no $i$-block of $\mathscr{A}(r,P)$ meets a boundary $i$-block of $\mathscr{A}(r+2iU)$. Upon dividing Inequality $\ref{eqn:squeeze}$ through by $t(r,P)$, letting $r \rightarrow \infty$, and applying the Normality Lemma, we arrive at the desired result.


A similar argument establishes Equation \ref{eqn:e}.\end{proof}

Substituting Equations \ref{eqn:v} and \ref{eqn:e} in to Equation \ref{eqn:Euler} yields the following result.

\begin{corollary} For an $i$-block transitive tiling whose $i$-blocks each have $n$ nodes (counted with multiplicity), we have the following Diophantine equation. \begin{equation} \sum_{j = 1}^{n}\frac{1}{\alpha_j} = \frac{n - 2i}{2} \label{eqn:Dio}\end{equation} \end{corollary}Note that since each pentagon has at least 5 nodes, we have $n \geq 5i$. Also, note that the left-hand side of Equation \ref{eqn:Dio} is maximized when $\alpha_j = 3$ for every $j$, which implies that $n \leq 6i$. 

\begin{corollary} For an $i$-block transitive tiling whose $i$-blocks each have $n$ nodes (counted with multiplicity), we have \begin{equation} 5i \leq n \leq 6i. \label{eqn:n-bounds}\end{equation}\end{corollary}

Inequality \ref{eqn:n-bounds} is nice as it establishes an upper bound on just how badly non-edge-to-edge an $i$-block transitive tiling can be. \emph{Consequently, for each positive integer $i$, there are only a finite number of types of convex pentagon that admit $i$-block transitive tilings.}

We note a few other interesting consequences of Theorem \ref{thm:main}.  Let $3_{\text{min}}$ and $3_{\text{max}}$ denote the minimum and maximum number of 3-valent nodes of an $i$-block with $n$ nodes (counted with multiplicity). Notice that the left-hand side of Equation \ref{eqn:Dio} is minimized when as few as possible of the $\alpha_j$'s are 3's, so $3_{\text{min}}$ is determined by solving the equation \[\frac{n - 2i}{2} = \frac{3_{\text{min}}}{3} + \frac{n-3_{\text{min}}}{4}\] for $3_{\text{min}}$, obtaining \begin{equation} 3_{\text{min}} = 3n - 12i. \label{eqn:3min}\end{equation} Similarly, since the number $k$ of 3-valent nodes in an $i$-block must satisfy \[\frac{k}{3} \leq \frac{n - 2i}{2},\] we see that \begin{equation} 3_{\text{max}} = \left \lfloor \frac{3n - 6i}{2} \right \rfloor. \label{eqn:3max}\end{equation}

We may make another observation concerning Equation \ref{eqn:Dio}: If $p$ is the average vertex valence, then \[\frac{n}{p} = \frac{n - 2i}{2},\] so \[p = \frac{2n}{2n-i},\] and by Inequality \ref{eqn:n-bounds}, we see that \begin{equation}3 \leq p \leq \frac{10}{3}. \label{eqn:p-bounds}\end{equation} $p = 3$ corresponds to the case that pentagons of $\mathscr{T}$ have on average 6 nodes (allowing for straight angles in non-edge-to-edge tilings by pentagons), and $p = 10/3$ corresponds to the case that the tiling is edge-to-edge (corresponding to a result in \cite{OB1}). This makes it clear that in any $i$-block transitive tiling, there will be some 3-valent nodes and (except when $p = 3$) some nodes with valence $k \geq 4$.

For specific values of $n$ and $i$, all solutions (for the $\alpha_j$) of Equation \ref{eqn:Dio} can be found. If $\alpha_1, \alpha_2, \ldots, \alpha_n$ is a solution, we will denote that solution by $\left<\alpha_1. \alpha_2.\ldots.\alpha_n\right>$ and call it an \emph{$(i,n)$-block species}. We will use exponents to indicate repeated values of $\alpha_i$. For example, in Figure \ref{fig:type-7-block}, the 2-block is of species $\left<4.3.3.4.4.4.3.3.3.3\right> = \left<4^4.3^6\right>$. In Table \ref{tab:Dio-solutions}, all $(i,n)$-block species are listed for $n \leq 3$.

\section{Possible Topological Types for $(i,n)$-block Species}

Let $\mathscr{T}$ be an $i$-block transtive tiling by congruent convex pentagons and let $\mathscr{I}$ be the corresponding isohedral tiling by $i$-block $\mathscr{B}$. Since $\mathscr{I}$ is isohedral, then it is one of 11 topological types, and from among these 11 topological types, the maximum vertex valence is 12 \cite{GS2}. Further, since at most $i$ pentagons meet at any node of $\mathscr{B}$, then in Equation \ref{eqn:Dio}, we must have \begin{equation} \alpha_j \leq 12i \label{eqn;max-valence}\end{equation} for all $j$. Inequality \ref{eqn;max-valence}† ensures that Equation \ref{eqn:Dio} has finitely many solutions for any $i$ and that these solutions can, for small values of $i$, be quickly found using a simple computer algorithm. The numbers in $\{V_{\mathscr{I}}(v)| v \text{ is a vertex of } \mathscr{I} \cap \mathscr{T}\}$ are exactly the numbers appearing in the topological type for $\mathscr{I}$, and this observation gives rise to the following facts.

\begin{lemma} A vertex $v$ of both $\mathscr{I}$ and $\mathscr{T}$ satisfies 

	\begin{enumerate}
	
	\item $V_{\mathscr{I}}(v) \leq V_{\mathscr{T}}(v)$\label{lem:maxval-general} 
	\item  $V_{\mathscr{T}}(v) \leq 3i$ if $\mathscr{I}$ has topological type $[3^6]$
	\item $V_{\mathscr{T}}(v) \leq  4i$ if $\mathscr{I}$ has topological types $[3^3.4^2]$, $[3^2.4.3.4]$, or $[4^4]$
	\item $V_{\mathscr{T}}(v) \leq  6i$ if $\mathscr{I}$ has topological types $[3^4.6]$, $[3.6.3.6]$, $[6^3]$,  or $[3.4.6.4]$
	\item $V_{\mathscr{T}}(v) \leq  8i$ if $\mathscr{I}$ has topological type $[4.8^2]$
	\item $V_{\mathscr{T}}(v) \leq  12i$ if $\mathscr{I}$ has topological type $[3.12^2]$
	
	\end{enumerate}
	
\label{lem:maxval} \end{lemma}

Referring to Figure \ref{fig:type-7-block}, we see that the inequality of Lemma \ref{lem:maxval}, Part \ref{lem:maxval-general} need not be an equality. The next results concerns those vertices of $\mathscr{T}$ that are not also vertices of $\mathscr{I}$; these vertices are in the interior of edges in $\mathscr{I}$, and as such, these vertices play a key role in how copies of $i$-blocks can meet in $\mathscr{I}$.

\subsection{Adjacency Conditions for $i$-Blocks}

Let $\mathscr{B}$ be an $i$-block for an $i$-block transitive tiling $\mathscr{T}$, and let $\beta_1, \beta_2, \ldots, \beta_n$ be the vertices of $\mathscr{T}$ on the boundary of $\mathscr{B}$, taken in order with respect to an orientation on $\mathscr{B}$. Let $b_i$ denote the number of pentagons of $\mathscr{B}$ that are incident with $\beta_i$. Then the \emph{boundary code} of $\mathscr{B}$ is the finite sequence \[\partial(\mathscr{B})=b_1b_2\ldots b_n.\] For example, the 2-block of Figure \ref{fig:type-7-block} has boundary code $\partial(\mathscr{B}) = 21112111$.

Because $\mathscr{B}$ is a prototile for isohedral tiling $\mathscr{I}$, then $\mathscr{B}$ has an associated incidence symbol that prescribes the manner in which copies of $\mathscr{B}$ are  surrounded by incident copies of $\mathscr{B}$. For example, if $\mathscr{I}$ is of isohedral type IH12, which has topological type $[3^6]$ and incidence symbol $[ab^{+}c^{+}dc^{-}b^{-};dc^{-}b^{-}a]$, then $\mathscr{B}$ tiles the plane as a topological hexagon, and its boundary is partitioned into 6 arcs that must match one another according to the incidence symbol (we refer the reader to \cite{GS1} or \cite{GS2} for an explanation of incidence symbols). A valid partition of the boundary of $\mathscr{B}$ must be compatible with this incidence symbol. The endpoints of the arcs forming the partition of the boundary of $\mathscr{B}$ will be indicated by placing over bars on the corresponding entries of $\partial(\mathscr{B})$; we will call a boundary code so marked a \emph{partitioned boundary code} and denote it by $\overline{\partial}(\mathscr{B})$. For example, the 3-block of Figure \ref{fig:2-block-tiling} has partitioned boundary code  \[\overline{\partial}(\mathscr{B}) = \bar{2}\bar{1}\bar{1}\bar{1}2\bar{1}\bar{1}1.\] The unmarked elements in $\overline{\partial}(\mathscr{B})$ correspond to the vertices of $\mathscr{T}$ on the boundary of $\mathscr{B}$ that are not vertices of $\mathscr{I}$. Thus, an edge of $\mathscr{B}$ of length $k$ corresponds to a subsequence of $\overline{\partial}(\mathscr{B})$ of the form \[e = \overline{b_i}b_{i+1}b_{i+2} b_{i+k-1}\overline{b_{i+k}}.\] As in the incidence symbols for the isohedral types, we will use superscripts to indicate the orientation of edges with respect to their mother tiles.

\begin{lemma}[The Matching Lemma] Let \[e_1 =\overline{b_i} b_{i+1} b_{i+2}\ldots b_{i+k-1}\overline{b_{i+k}} \text{ and } e_2 = \overline{b_j} b_{j+1} b_{j+2}\ldots b_{j+k-1}\overline{b_{j+k}}\] be two length $k$ edges on the boundary of $\mathscr{B}$, allowing for the possibility that $e_1 = e_2$. 
\begin{enumerate}
\item $e_{1}^{+}$ may meet $e_{2}^{+}$ (or $e_{1}^{-}$ may meet $e_{2}^{-}$) if $b_{i+t} + b_{j+k-t} = V_{\mathscr{T}}(\beta_{i+t}) = V_{\mathscr{T}}(\beta_{j+k-t})$ for each integer $t$, $1 \leq t \leq k-1$. 
\item $e_{1}^{+}$ may meet $e_{2}^{-}$ if $b_{i+t} + b_{j+t} = V_{\mathscr{T}}(\beta_{i+t})$ for each integer $t$, $1 \leq t \leq k-1$. 
\item $e_{1}$ may meet $e_{2}$ if both of the previous two conditions hold.
\item ($1$s cannot meet $1$s) In particular, in the case that $e_{1}^{+}$ meets $e_{2}^{+}$, we must have $b_{i+t} + b_{j+k-t} \geq 3$,  so it can never be the case that $b_{i+t} = 1 = b_{j+k-t}$. Similarly, in the case that $e_{1}^{+}$ meets $e_{2}^{-}$, it never happen that $b_{i+t} = 1 = b_{j+t}$.\label{1s-no-match}
\item (Interior vertices cannot be too large) For each vertex $\beta$ in the interior of an edge on the boundary of $\mathscr{B}$, $V_{\mathscr{T}}(\beta) \leq 2i$. \label{bigint}
\end{enumerate}
\label{lem:matching} \end{lemma}

Because any vertex of the boundary of $\mathscr{B}$ must be matched with at least one other vertex on an adjacent copy of $\mathscr{B}$, the Matching Lemma implies the following result, which can be used to eliminate possible topological types for a given $(i,n)$-block species.
	
\begin{lemma} Let $\mathscr{B}$ be of $(i, m_1 + m_2 + \cdots + m_k)$-block species type $\left<\alpha_1^{m_1}. \alpha_2^{m_2}. \, \ldots \, .\alpha_k^{m_k}\right>$.

	\begin{enumerate}
	\item If the boundary of $\mathscr{B}$ contains a vertex $v_{i}$ with $V_{\mathscr{T}}(v_{i}) = \alpha_p > 2i$ and $m_p = 1$, then the topological type of $\mathscr{I}$ must contain the number $\alpha_p$. \label{solo}
	\item If the boundary of $\mathscr{B}$ contains vertices $v_{i} \neq v_{j}$ with $\alpha_p = V_{\mathscr{T}}(v_{i}) = V_{\mathscr{T}}(v_{j}) > 2i$ and $m_p = 2$, then the topological type of $\mathscr{I}$ must contain the number $\alpha_p$ twice. \label{nonfused}
	\item If the boundary of $\mathscr{B}$ contains a vertex $v_{i}$ that is incident with 2 pentagons of $\mathscr{B}$, $\alpha_p = V_{\mathscr{T}}(v_{i}) > 2i$,  and $m_p = 2$, then the topological type of $\mathscr{I}$ must contain the number $\alpha_{p}/2$.\label{fused}
	
	\end{enumerate}
	 \label{lem:highisolated} \end{lemma}
	  
Lemmas \ref{lem:maxval}, \ref{lem:matching}, and \ref{lem:highisolated} can be used to eliminate several topological types for a given $(i,n)$-block species.    For example, for the $(1,5)$-block species $\left<3^3.4^2\right>$, Lemma \ref{lem:maxval} Part \ref{lem:maxval-general} says possible topological types for $\mathscr{I}$ contain at most three 3s, at most two 4s, and no other numbers. This leaves only $[3^3.4^2]$ and $[3^2.4.3.4]$. For the $(1,5)$-block species $\left<3^4.6\right>$, the only possible topological type for $\mathscr{I}$ is $[3^4.6]$. 

In a similar way, we can eliminate possible topological types corresponding to larger values of $i$. Consider the $(3,15)$-block species $\left<3^{13}.8.24\right>$. The vertex of valence 24 very much restricts the possible topological types for $\mathscr{I}$; since  $24>6 \cdot 3$, Lemma \ref{lem:maxval} says that no vertex of $\mathscr{T}$ can have valence 24 and simultaneously be a vertex of $\mathscr{I}$ unless the topological type of $\mathscr{I}$ has a vertex of valence 8 or greater. Further, since $24 > 2 \cdot 3$, Lemma \ref{lem:matching} Part \ref{bigint} guarantees that no vertex in $\mathscr{T}$ but not in $\mathscr{I}$ can have valence 24. Thus, $\mathscr{I}$ cannot be of topological types $[3^6]$, $[3^3.4^2]$, $[3^2.4.3.4]$, $[4^4]$, $[3^4.6]$, $[3.6.3.6]$, $[6^3]$,  or $[3.4.6.4]$. Thus, in any $3$-block transitive tiling of species type $\left<3^{13}.8.24\right>$, the only possible topological types are $[4.8^2]$ and $[3.12^2]$. But, using Lemma \ref{lem:highisolated} Part \ref{solo}, we can eliminate both of these two remaining topological types since neither of these topological types contains 24. 

As another example, consider the $(4,20)$-block species $\left<3^{17}.5.15^2\right>$. Since $15>2 \cdot 4$, Lemma \ref{lem:highisolated}  Parts \ref{nonfused} and \ref{fused} implies that the permissible topological types for $\mathscr{I}$ must contain 15 twice or 15/2. Notice that there are no topological types satisfying these conditions.  

We provide one last lemma that relates partitions of the boundary of an $i$-block to corresponding possible topological types for the $i$-block.

\begin{lemma}  Let $\Delta = \#1's - \#\text{non-}1's$ in $\partial(\mathscr{B})$.
\begin{enumerate}
\item If $\Delta > 6$, $\mathscr{B}$ does not admit a tile-transitive tiling of the plane.
\item If $\Delta = 6$, $\mathscr{B}$ admits only tile-transitive tilings of topological type $[3^6]$, and every marked element of $\overline{\partial}(\mathscr{B})$ is a 1.
\item If $\Delta = 5$, $\mathscr{B}$ admits only tile-transitive tilings of hexagonal or pentagonal topological types, and every marked element of $\overline{\partial}(\mathscr{B})$ is a 1.
\item If $\Delta= 4$, $\mathscr{B}$ admits only tile-transitive tilings of hexagonal, pentagonal, or quadrilateral topological types. For pentagonal and quadrilateral types, every marked element of $\overline{\partial}(\mathscr{B})$ is a 1, and for hexagonal types, five 1's of $\overline{\partial}(\mathscr{B})$ must be marked.
\end{enumerate} \label{lem:1-imbal} \end{lemma}

Lemma \ref{lem:1-imbal} is useful in a few ways. First, for a particular generalized $(i,n)$-block, we may (at a glance) eliminate certain possible topological types from consideration. Secondly, this lemma drastically limits the number of ways that $\partial(\mathscr{B})$ can be partitioned.

In Table \ref{tab:Dio-solutions}, we have organized the $(i,n)$-block species and the corresponding possible topological types for $i \leq 3$.

\begin{table}[H]\centering\tiny
\begin{tabular}{|l|l|r|}\hline
$(i,n)$ & $(i,n)$-block species & possible topological types for $\mathscr{I}$  \\ \hline\hline

$(1,5)$ & $\left<3^3.4^2\right>$ & $[3^3.4^2],[3^2.4.3.4]$  \\
       & $\left<3^4.6\right>$ & $[3^4.6]$     \\\hline

  $(1,6)$ & $\left<3^6\right>$ & $[3^6]$    \\\cline{1-3}

$(2,10)$ & $\left<3^8.4.12\right>$ & -      \\
      & $\left<3^8.6^2\right>$ & $[3^6],[3^4.6],[3.6.3.6]$     \\
      & $\left<3^7.4^2.6\right>$ & $[3^4.6]$     \\
      & $\left<3^6.4^4\right>$& $[3^6],[3^3.4^2],[3^2.4.3.4],[4^4]$    \\\cline{1-3}

$(2,11)$ & $\left<3^{10}.6\right>$& $[3^4.6]$     \\
& $\left<3^9.4^2\right>$& $[3^6],[3^3.4^2],[3^2.4.3.4]$     \\\cline{1-3}

$(2,12)$ &$\left<3^{12}\right>$ & $[3^6]$    \\\cline{1-3}

$(3,15)$ & $\left<3^{13}.7.42\right>$, $\left<3^{13}.8.24\right>$, $\left<3^{13}.9.18\right>$, $\left<3^{13}.10.15\right>$ & -    \\
& $\left<3^{12}.4.5.20\right>$, $\left<3^{12}.5^2.10\right>$, $\left<3^{11}.4^3.12\right>$ & -     \\
& $\left<3^{13}.12^2\right>$ &  $[3^4.6],[3.12^2]$     \\
& $\left<3^{12}.4.6.12\right>$ & $[4.6.12]$    \\
& $\left<3^{12}.4.8^2\right>$ & $[3^3.4^2],[3^2.4.3.4],[4.8^2]$    \\
& $\left<3^{12}.6^3\right>$ & $[3^6],[3^3.4^2],[3^2.4.3.4],[3^4.6],[3.6.3.6],[6^3]$     \\
& $\left<3^{11}.4^2.6^2\right>$ & $[3^6],[3^3.4^2],[3^2.4.3.4],[4^4],[3^4.6],[3.6.3.6],[3.4.6.4]$    \\
& $\left<3^{10}.4^4.6\right>$ & $[3^6],[3^3.4^2],[3^2.4.3.4],[4^4],[3^4.6]$     \\
& $\left<3^{9}.4^6\right>$ & $[3^6],[3^3.4^2],[3^2.4.3.4],[4^4]$    \\\cline{1-3}

$(3,16)$ & $\left<3^{14}.4.12\right>$ & -    \\
& $\left<3^{14}.6^2\right>$ & $[3^6],[3^3.4^2],[3^2.4.3.4],[3^4.6],[3.6.3.6]$     \\
& $\left<3^{13}.4^2.6\right>$ & $[3^6],[3^3.4^2],[3^2.4.3.4],[3^4.6], [3.4.6.4]$     \\
& $\left<3^{12}.4^4\right>$ & $[3^6],[3^3.4^2],[3^2.4.3.4],[4^4]$     \\\cline{1-3}

$(3,17)$ & $\left<3^{16}.6\right>$ & $[3^6],[3^4.6]$   \\
& $\left<3^{15}.4^2\right>$ & $[3^6],[3^3.4^2],[3^2.4.3.4]$   \\\cline{1-3}

  $(3,18)$ & $\left<3^{18}\right>$ & $[3^6]$     \\ \hline

\end{tabular}\caption{All $(i,n)$-block species for $1 \leq i \leq 3$}\label{tab:Dio-solutions}\end{table}

\section{An algorithm for enumerating all pentagons admiting $i$-block transitive tilings.}

For fixed $i$, the following procedure will determine all possible systems of equations corresponding to $i$-block transitive tilings.
\begin{enumerate}
\item Enumerate all topological $i$-block forms with $n$ nodes (subject to the restriction that $5i \leq n \leq 6i$ from Inequality \ref{eqn:n-bounds}. This part of the procedure was done by hand for $i = 1, 2, 3$, and $4$.
\item For each topological $i$-block with assigned flat nodes, generate every possible labeling of the consituent pentagons' angles and sides with $A, \ldots, E$ and $a, \ldots, e$.
\item In every way possible, assign the value of $\pi$ to nodes in the pentagons of the $i$-block form having more than 5 nodes, leaving each pentagon with exactly 5 unlabeled nodes.
\item For each $i$-block form, generate every partition of the boundary into 3, 4, 5, or 6 consecutive arcs.
\item For each such bounary partition, determine all compatible isohedral types.
\item For each fully-labeled topological $i$-block, apply the adjacency symbol of each compatible isohedral type to the partition in every way possible.
\item For each application of the adjacency symbol, generate the corresponding set of linear equations relating the sides and angles of the pentagons of the $i$-block and determine if this system of equations is consistent. Any inconsistent linear systems are discarded.
\item For each consistent system, determine whether or not the resulting system of equations implies that the pentagon is of a previously known type.
\item For any system of equations not identified as a previously observed type, determine if a pentagon satisfying these equations is geometrically realizable. That is, determine whether or not such a pentagon can additionally satisfy the system of equations consponding to a 0 vector sum for the sides under the constraint of positive side lengths and angle measure strictly between 0 and $\pi$. 
\end{enumerate}

We will illustrate process for a sample $2$-block and, separately, a sample $3$-block. While all of our results for enumerating pentagons admitting 1-, 2-, and 3-block transitive were determined by a single automated system (except parts corresponding to steps 1 and 9 above), as a double-check on our automated algorithm, we separately enumerated the pentagons admitting 1-block transitive tilings completely by hand (\ref{app:1-blocks}), and we separately enumerated the pentagons admitting 2-block transtive tilings partially by hand and partially using Mathematica code to automate the label applications and the linear system solving.

\subsection{Illustrating the algorithm with a 2-block example}

To facilitate discussion of 2-blocks in general, we will use regular shapes to represent the pentagons comprising the 2-blocks, even though in any actual 2-block, the two (congruent) pentagons are irregular. In representing 2-blocks in a generalized way makes spotting flat nodes visually apparent, we will represent pentagons having 5 nodes as regular pentagons, pentagons having 6 nodes (1 flat node) will be represented by regular hexagons, and pentagons having 7 nodes (2 flat nodes) will be represented by regular heptagons. By Inequality \ref{eqn:n-bounds} the number of nodes $n$ (counted with multiplicity) in a 2-block satisfies $10 \leq n \leq 12$, so there are 4 ways to represent 2-blocks in terms of the numbers of nodes; these are depicted in Figure \ref{fig:2-blocks}. In these topological 2-blocks forms, it is important to note that in any hexagon, one of the corners must represent a flat node, and in any heptagon, 2 of the corners must represent flat nodes.

\begin{figure}[H]
\centering
\begin{subfigure}[H]{.24\textwidth} 
\centering
\includegraphics[scale=.2]{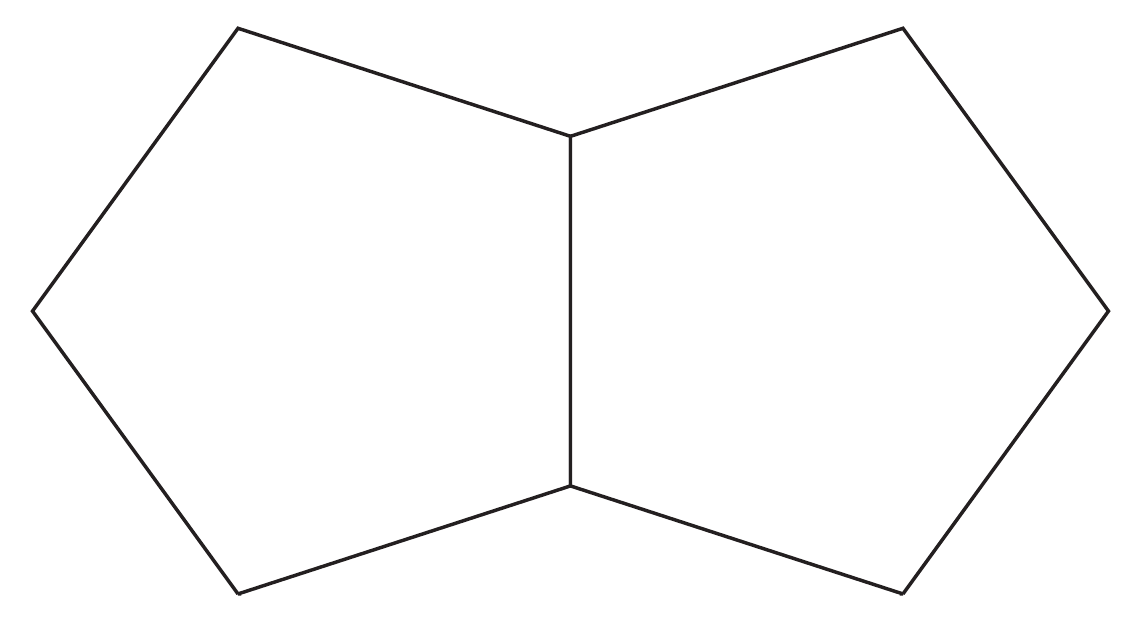}  
\caption{$n = 10$\\$\partial(\mathscr{B}) = 21112111$}\label{fig:pent-pent}
\end{subfigure}
\begin{subfigure}[H]{.24\textwidth}
\centering
\includegraphics[scale=.2]{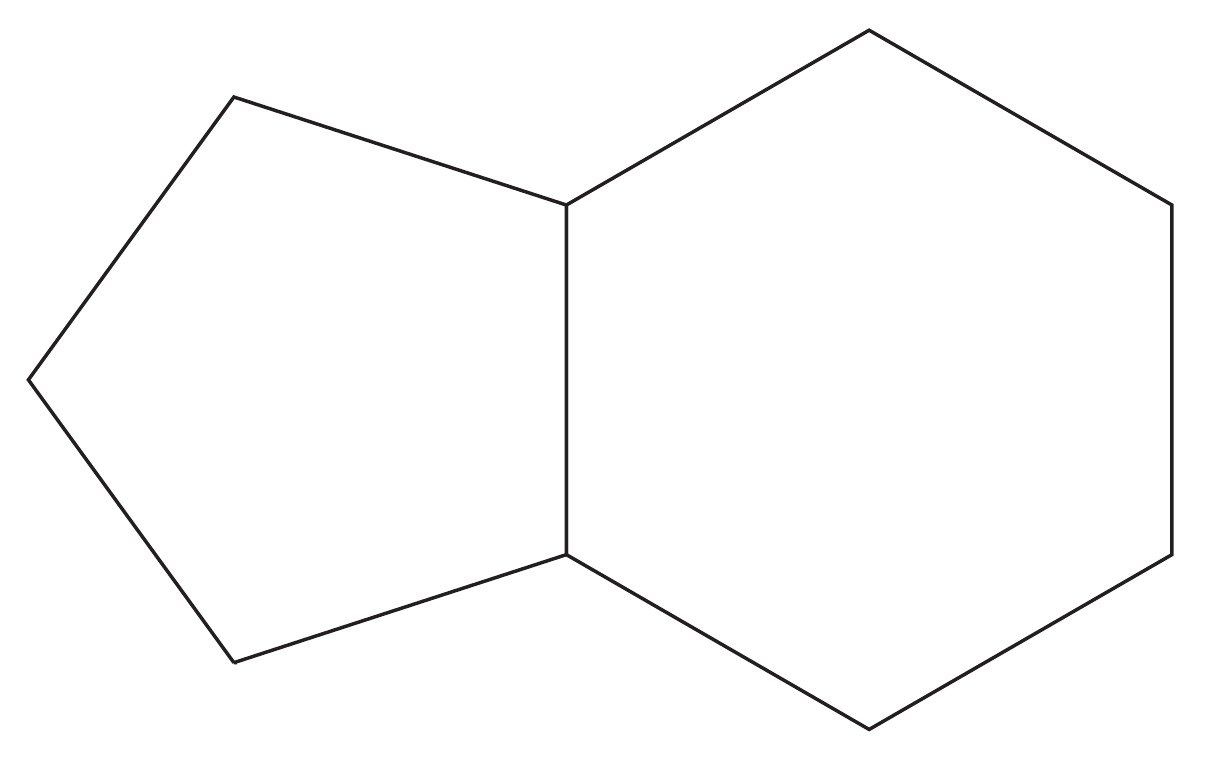}  
\caption{$n = 11$\\$\partial(\mathscr{B}) = 211121111$}\label{fig:pent-hex}
\end{subfigure}
\begin{subfigure}[H]{.24\textwidth}
\centering
\includegraphics[scale=.2]{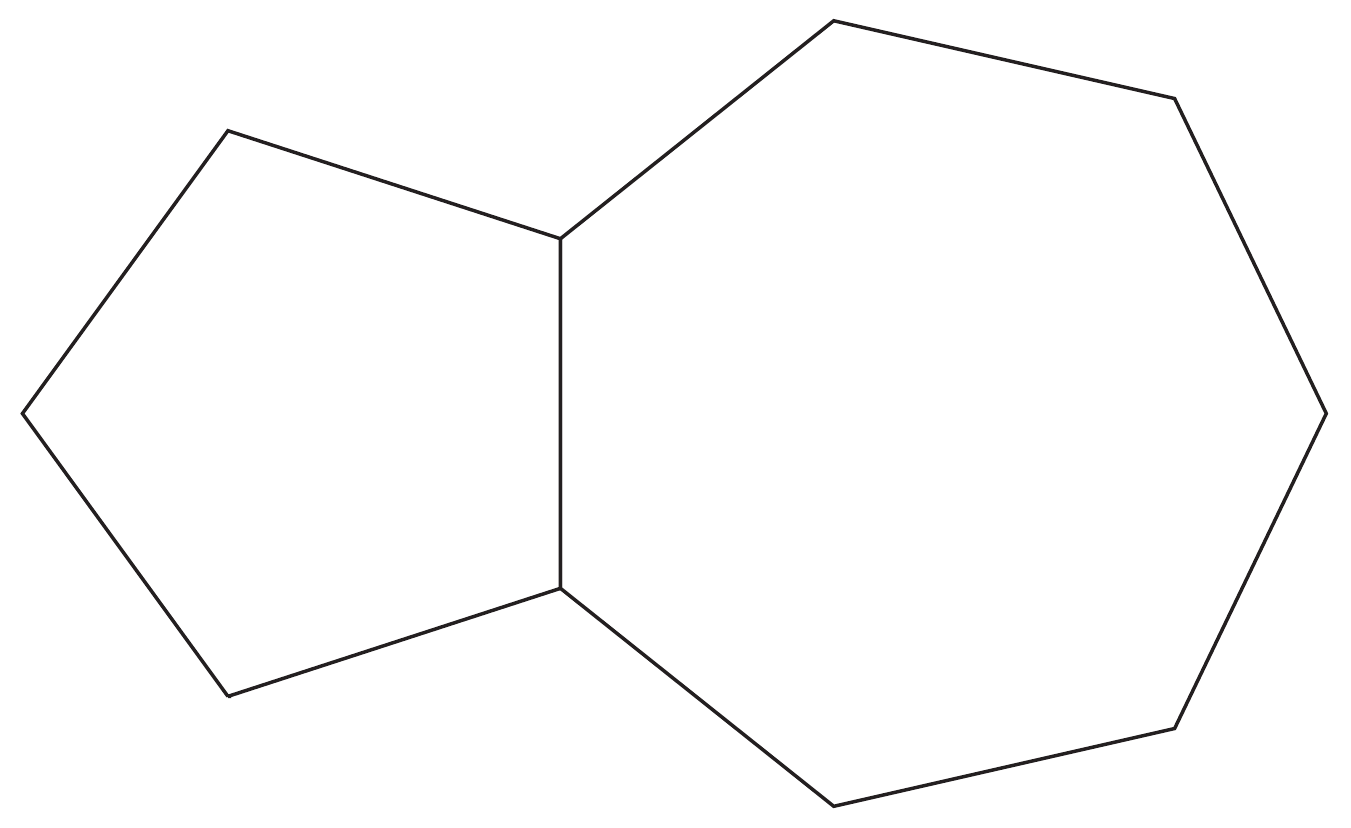}  
\caption{$n = 12$\\$\partial(\mathscr{B}) = 2111211111$}\label{fig:pent-hept}
\end{subfigure}
\begin{subfigure}[H]{.24\textwidth}
\centering
\includegraphics[scale=.2]{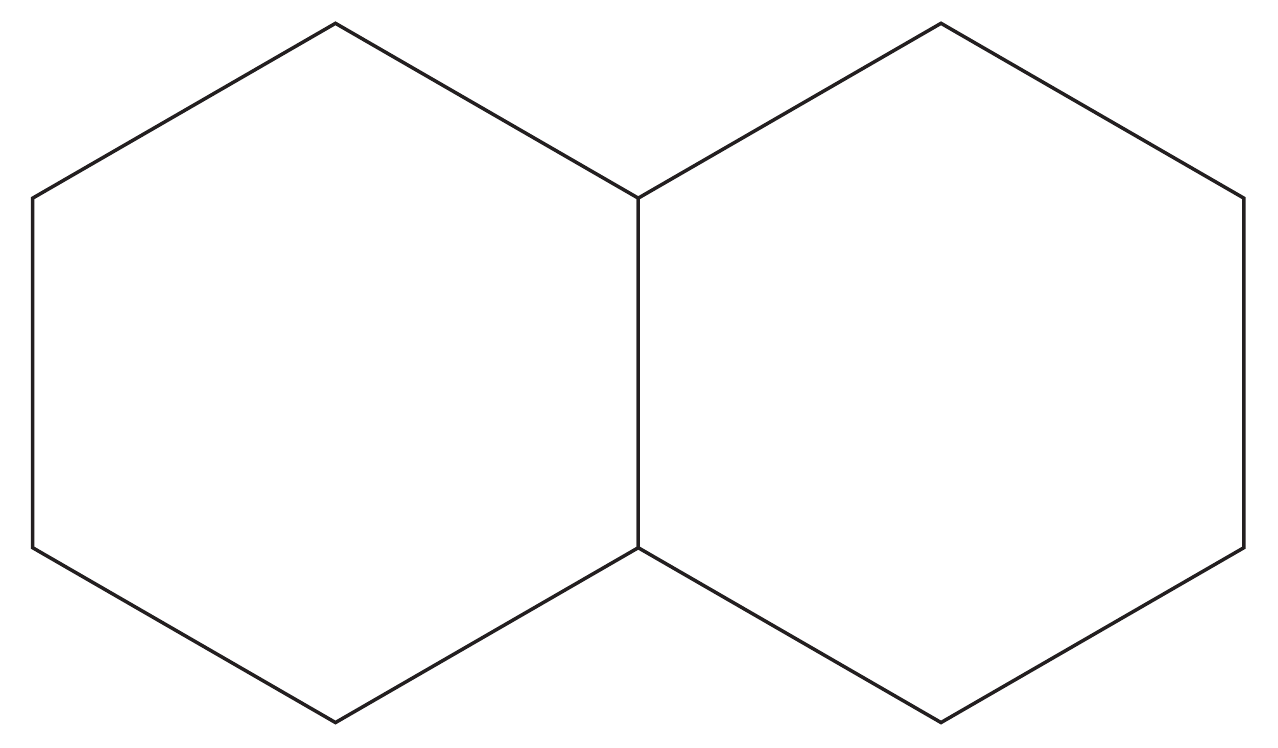} 
\caption{$n = 12$\\$\partial(\mathscr{B}) = 2111121111$} \label{fig:hex-hex}
\end{subfigure}
\caption{All possible topological 2-blocks forms}\label{fig:2-blocks}
\end{figure}

Now, to illustrate the algorithm outlined above, for step 1, let us pick the topological $i$-block form above represented in Figure \ref{fig:pent-hept}. For steps 2 and 3, without loss of generality, label the vertices of the left pentagon of Figure \ref{fig:IH6-2-block-unlabeled} with $A$, $B$, $C$, $D$, and $E$. The right pentagon, however, may be in several different orientations with respect to the choice of labeling of the first pentagon. We choose variable labels $T$, $U$, $V$, $W$, $X$, $Y$, and $Z$ for the nodes of this second pentagon (Figure \ref{fig:IH6-2-block-unlabeled}). 

\begin{figure}[H]\centering \includegraphics[scale=.7]{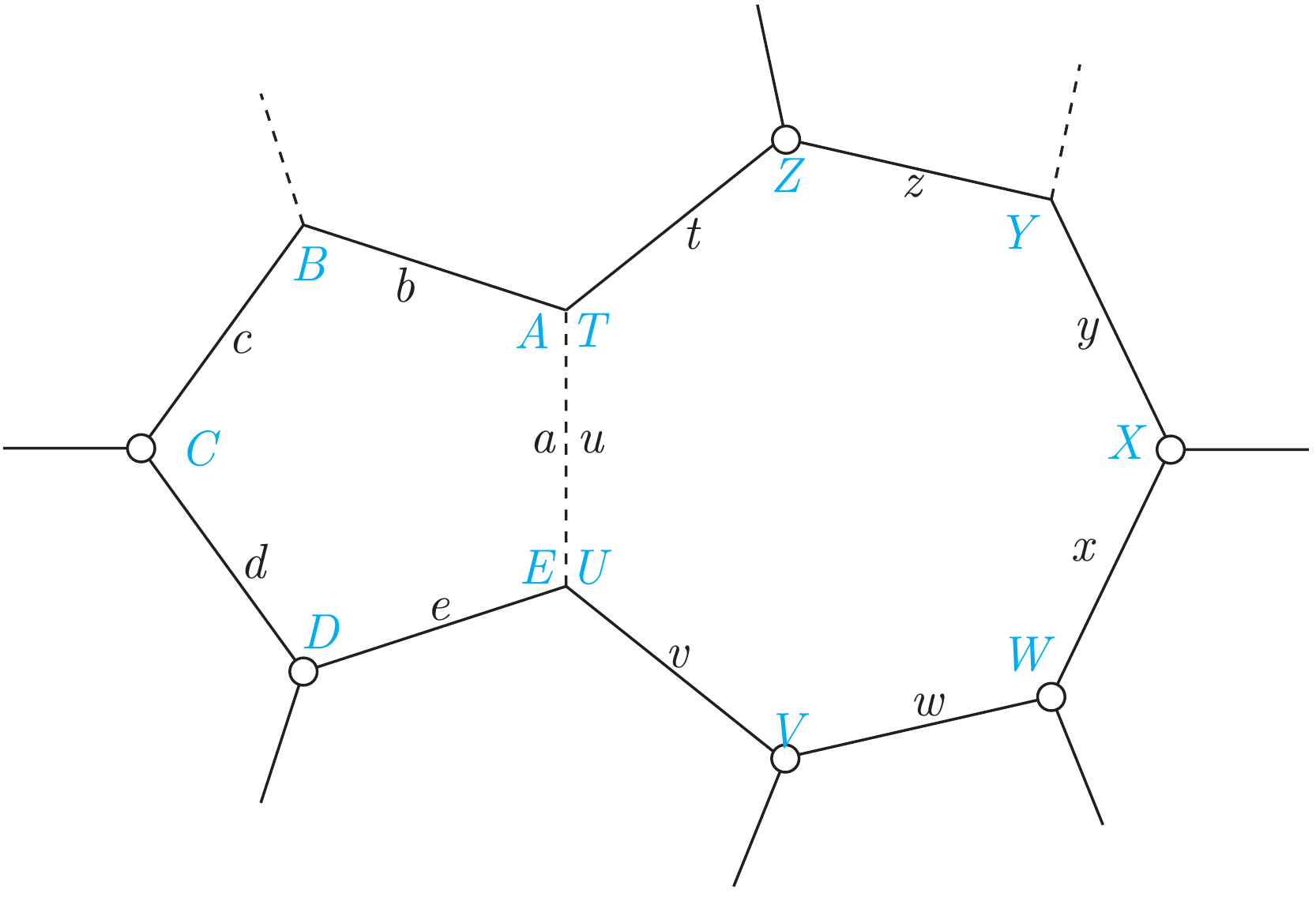}  \caption{Partially labeled 2-block} \label{fig:IH6-2-block-unlabeled}\end{figure} 

\noindent These variables may assume the values $A$, $B$, $C$, $D$, $E$, or $\pi$ (two of the nodes on the heptagon is a flat node). For example, the substitution $T=A$, $U=B$, $V = C$, $W = D$, $X = E$, $Y = \pi$, and $Z = \pi$ yields the labeling of nodes in Figure \ref{fig:IH6-2-block-sub}. 

For step 4, notice that the boundary code for the 2-block in this case is $2111211111$, for which $\Delta$ from Lemma \ref{lem:1-imbal} is $\Delta = 6$. By Lemma \ref{lem:1-imbal}, such a 2-block can admit isohedral tilings of hexagonal types only, and every marked element of $\overline{\partial}(\mathscr{B})$ must be a 1. Note also that two consecutive 1's cannot occur in the interior of a boundary edge of $\mathscr{B}$. After using Lemma \ref{lem:1-imbal} and our previous observation to filter out unusable boundary partitions, we are left with the boundary partitions in Table \ref{tab:2,12-pent-hept-blocks}, completing step 4 of the algorithm.

\begin{table}[H]\centering
\begin{tabular}{|l|l|l|}\hline\rule{0pt}{3ex} 

$2\bar{1}\bar{1}\bar{1}2\bar{1}\bar{1}1\bar{1}1$ & $2\bar{1}1\bar{1}2\bar{1}\bar{1}1\bar{1}\bar{1}$  &   $2\bar{1}1\bar{1}2\bar{1}1\bar{1}\bar{1}\bar{1}$  \\

$2\bar{1}\bar{1}\bar{1}2\bar{1}1\bar{1}\bar{1}1$ &   $2\bar{1}\bar{1}12\bar{1}\bar{1}\bar{1}1\bar{1}$      &   $2\bar{1}1\bar{1}21\bar{1}\bar{1}\bar{1}\bar{1}$   \\

$2\bar{1}\bar{1}\bar{1}2\bar{1}1\bar{1}1\bar{1}$ & $2\bar{1}\bar{1}12\bar{1}\bar{1}1\bar{1}\bar{1}$        &   $21\bar{1}\bar{1}2\bar{1}\bar{1}\bar{1}\bar{1}1$   \\

$2\bar{1}\bar{1}\bar{1}21\bar{1}\bar{1}\bar{1}1$ & $2\bar{1}\bar{1}12\bar{1}1\bar{1}\bar{1}\bar{1}$        & $21\bar{1}\bar{1}2\bar{1}\bar{1}\bar{1}1\bar{1}$ \\

$2\bar{1}\bar{1}\bar{1}21\bar{1}\bar{1}1\bar{1}$ &  $2\bar{1}\bar{1}121\bar{1}\bar{1}\bar{1}\bar{1}$        & $21\bar{1}\bar{1}2\bar{1}\bar{1}1\bar{1}\bar{1}$ \\

$2\bar{1}\bar{1}\bar{1}21\bar{1}1\bar{1}\bar{1}$  & $2\bar{1}1\bar{1}2\bar{1}\bar{1}\bar{1}\bar{1}1$ & $21\bar{1}\bar{1}2\bar{1}1\bar{1}\bar{1}\bar{1}$  \\
 
$2\bar{1}\bar{1}12\bar{1}\bar{1}\bar{1}\bar{1}1$  &$2\bar{1}1\bar{1}2\bar{1}\bar{1}\bar{1}1\bar{1}$  & $21\bar{1}\bar{1}21\bar{1}\bar{1}\bar{1}\bar{1}$ \\

& & $21\bar{1}12\bar{1}\bar{1}\bar{1}\bar{1}\bar{1}$\\\hline

\end{tabular}\caption{boundary partitions for pentagon-heptagon $(2,12)$-blocks}\label{tab:2,12-pent-hept-blocks}\end{table}

For our example, let us pick the partitioned boundary code $\overline{\partial}(\mathscr{B}) = 21\bar{1}\bar{1}2\bar{1}\bar{1}\bar{1} 1 \bar{1}$. In Figure \ref{fig:IH6-2-block-unlabeled} we have indicated this partition by putting white dots on the nodes marking the end points of the partition edges. For step 5, we must determine which isohedral types are compatible with this partition. The compatible isohedral types are determined by comparing the edge lengths in $\overline{\partial}(\mathscr{B})$ to the edge transitivity classes required for the isohedral types, as well as by applying the Matching Lemma. In doing this, we find that the compatible isohedral types are IH4, IH5, and IH6. Each compatible isohedral type will in turn be checked, but to illustrate our method, let us suppose our blocks form an IH6 tiling. The adjacency symbol for IH6 is $[a^{+}b^{+}c^{+}d^{+}e^{+}f^{+};a^{+}e^{-}c^{+}f^{-}b^{-}d^{-}].$ For step 6, we apply the IH6 adjacency symbol in every possible way to this labeled 2-block, as indicated by the red arcs labeled with Greek characters in Figure \ref{fig:IH6-2-block-sub}. In this case, there is only one way to apply the adjacency symbol.

\begin{figure}[H]\centering \includegraphics[scale=.7]{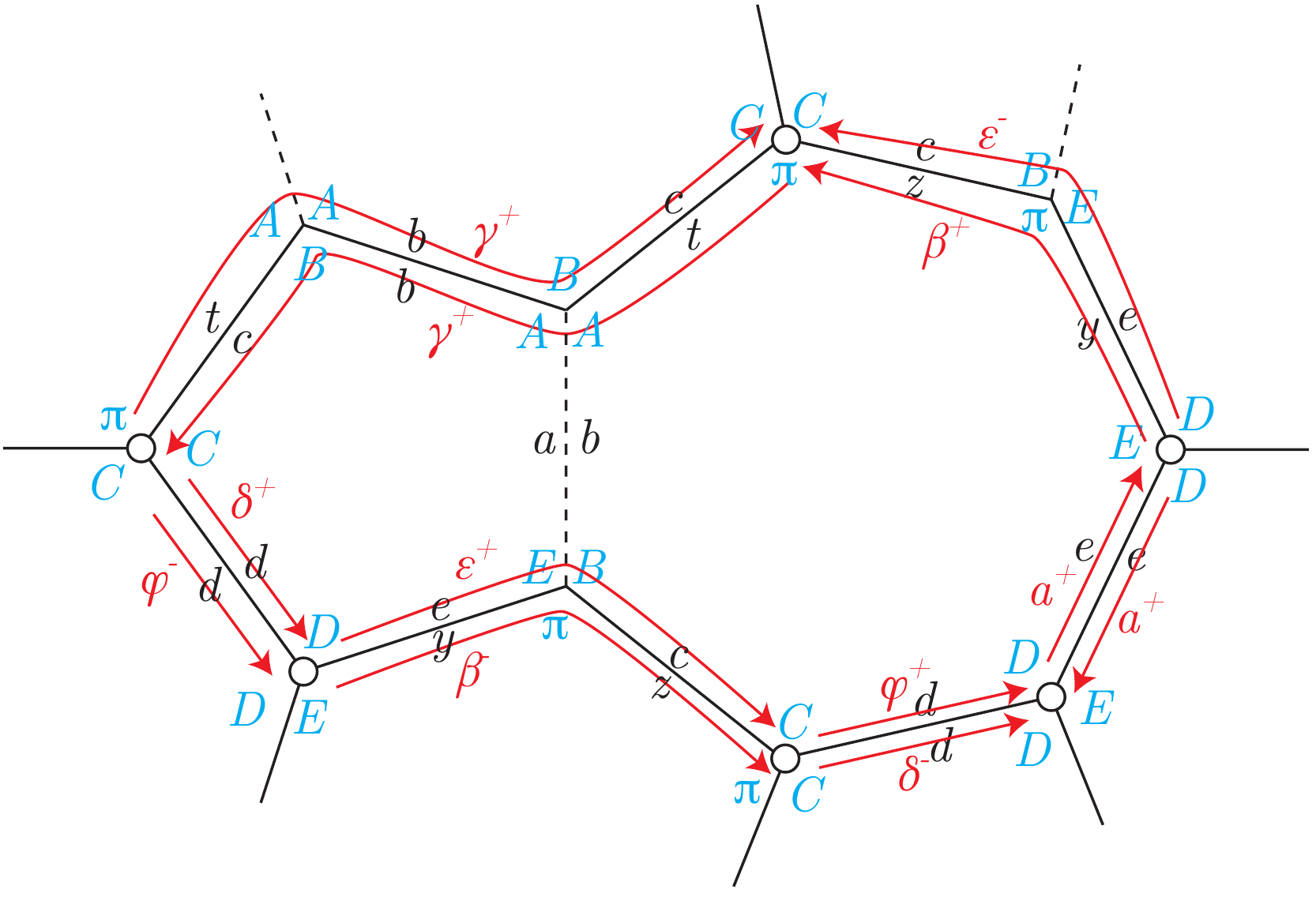}  \caption{IH6 labeling of a 2-block after substitutions} \label{fig:IH6-2-block-sub}\end{figure} 

\noindent For step 7, from Figure \ref{fig:IH6-2-block-sub}, the following system of equations are gleaned.
\begin{align*}
2A+B &= 2\pi\\
2C+ \pi &= 2\pi\\
2D + E &= 2\pi\\
E+B+\pi &=2\pi\\
a &= b \\
c+e &= y+z\\
2c+e &=t+y+z \\
a &=t+y+z\\\end{align*} Finally, for step 8, upon simplfying the equations and eliminating the variables $t, u, \ldots, z$, we arrive at the set of equations \begin{align*}
2A+B &=2\pi \\
C &= \pi/2\\
D &= 3\pi/2-A\\
a &= b\\
e &= a - 2c.\\\end{align*} Any pentagon admitting such a 2-block is then quickly identified as a Type 11 pentagon.

\subsection{Illustrating the algorithm with a 3-block example}

For step 1 of our algorithm for finding all convex pentagons admitting 3-block transitive tilings, we determine all of the possible topological 3-block forms. This part of the process was done by hand. In Figure \ref{fig:3-block-forms}, we show all possible topological 3-block forms (up to rotation and reflection).

\begin{figure}[H]\centering \includegraphics[scale=.7]{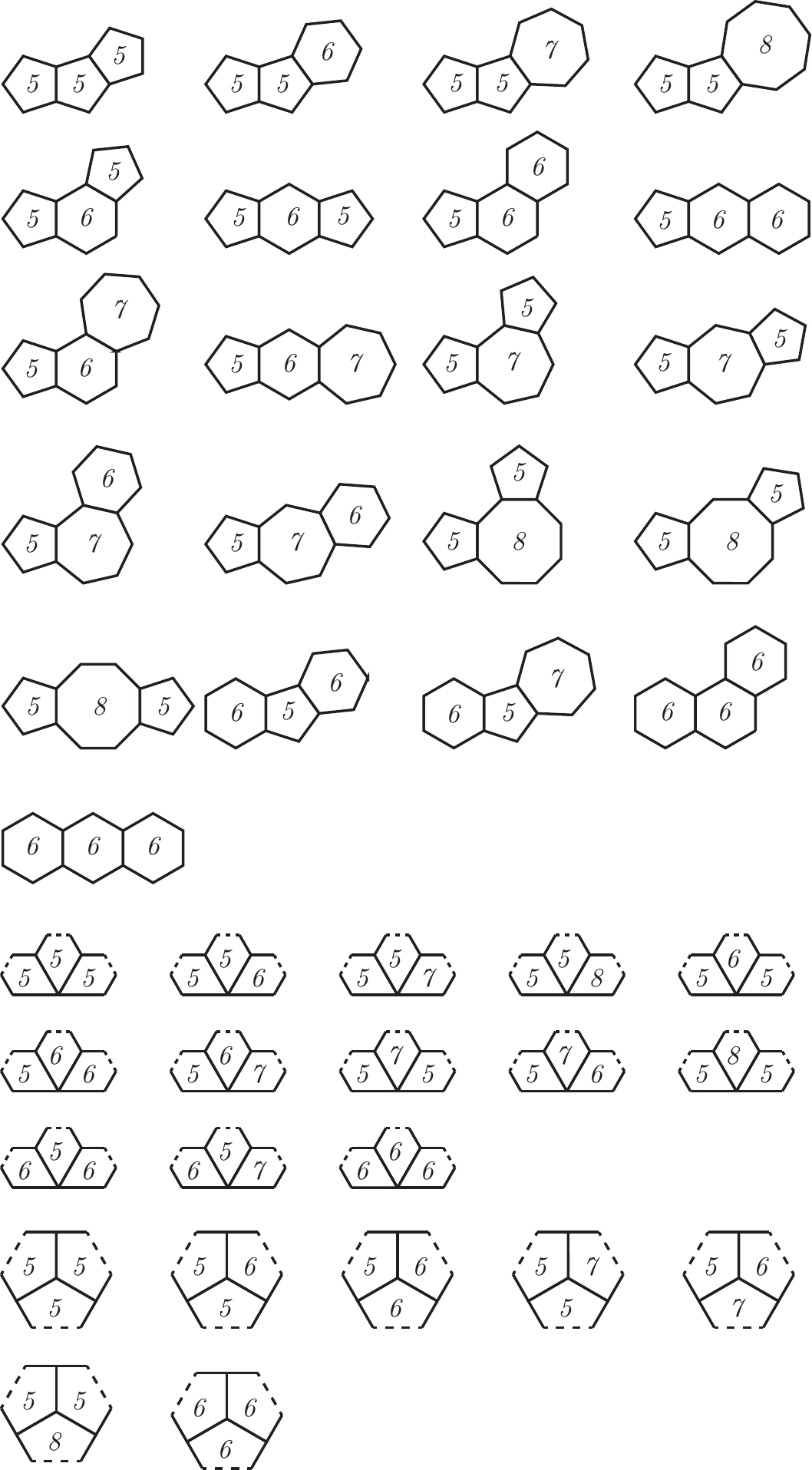}  \caption{All topological 3-block forms. The number labeling a polygon represents the number of nodes of that polygon.} \label{fig:3-block-forms}\end{figure} 

To illustrate subsequent steps of the algorithm, let us choose the 3-block form of Figure \ref{fig:556-rose-labeled} which has boundary code 1121121112 (starting at the top node and going counterclockwise). For steps 2 and 3, we must assign labels and $\pi$ nodes in every possible way to the nodes of this block. One such way of doing so is shown in Figure \ref{fig:556-rose-labeled}.

\begin{figure}[H]\centering \includegraphics[scale=1]{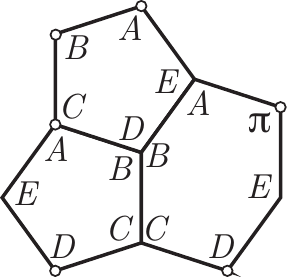}  \caption{A labeled 3-block form.} \label{fig:556-rose-labeled}\end{figure} 

For step 4, we must partition the boundary of this 3-block form, in every possible way, into 3, 4, 5, and 6 arcs. Since there are 10 sides on the boundary of this 3-block form, partitioning the boundary corresponds to finding all cyclically equivalent integer partitions of the integer 10 into 3, 4, 5, and 6 integers. For example, consider the integer partition $\left\{1,1,2,2,2,2\right\}$ of 10; this integer partition gives the number of sides per boundary edge in a partition of the boundary into 6 arcs. Applying this integer partition, we obtain the partitioned boundary code $\bar{1}\bar{1}\bar{2}1\bar{1}2\bar{1}1\bar{1}2$. In Figure \ref{fig:556-rose-labeled}, the vertices labeled with white dots indicate the endpoints of the edges forming this partition of the boundary that we will use to illustrate subsequent steps of the algorithm.

For step 5, we determine that the isohedral types compatible with this partition are IH2, IH5, IH7, IH15, and IH16. Performing step 6, we choose isohedral type IH5 and apply the adjacency symbol, $[a^{+}b^{+}c^{+}d^{+}e^{+}f^{+};a^{+}e^{+}d^{-}c^{-}b^{+}f^{+}]$, in every possible way. In this particular example, there is a unique way to apply the adjacency symbol (up to symmetry), as shown in Figure \ref{fig:556-rose-IH5}.

\begin{figure}[H]\centering \includegraphics[scale=1]{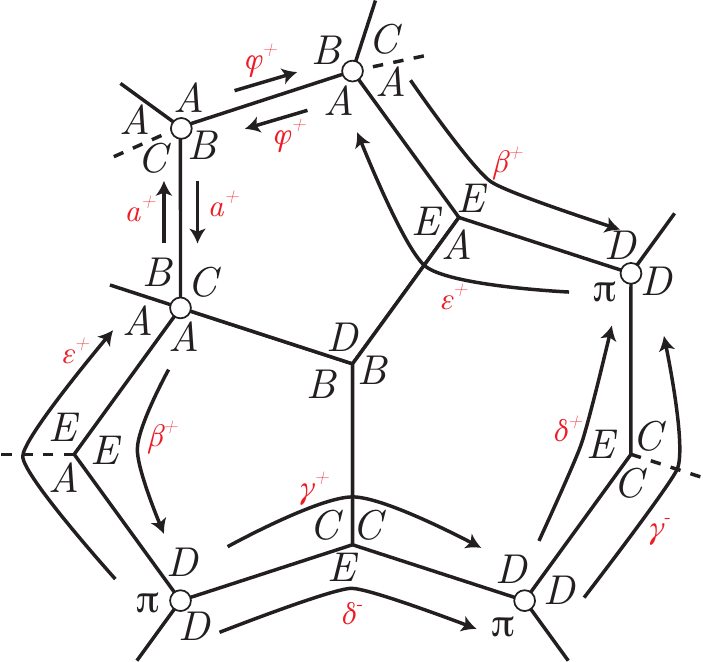}  \caption{A labeled IH5 3-block.} \label{fig:556-rose-IH5}\end{figure} 

For step 7, we simply read off the equations for the angles and sides from Figure \ref{fig:556-rose-IH5} to get the following system of equations.

\begin{align*} 2A+B+C &= 2p\\
2E+A &=2\pi\\
2D+\pi &= 2\pi\\
2C+E &= 2\pi\\
2B+D &= 2\pi\\
e = b &= d\\
a &= e+d\\\end{align*}

Upon simplifying this system, we obtain \begin{align*}
A &= \pi/3\\
B &= 3\pi/4\\
C &= 7\pi/12\\
D &= \pi/2\\
E &=5\pi/6\\
a &= 2b = 2d = 2e\\\end{align*}

For step 8, upon comparing this system to the previously known 14 types and any sets of equations we have previously identified as impossible, we do not find a match. This leads us to step 9: We must determine if this set of equations can be realized by a convex pentagon, and if new information is learned about the side and angle relations in the process, we must check if this new information yields a known type of pentagon. To test if these equations can be realized by a pentagon, we view the edges of a hypothetical pentagon satisyfing these equations as vectors and require that the sum of these vectors be 0. This results in a system of two equations: \footnotesize \begin{align} a - b \cos A + c \cos(A+B) - d \cos(A+B+C) + e \cos(A+B+C+D) &=0 \label{eqn:cos}\\ 
 b \sin A - c \sin(A+B)+ d \sin(A+B+C) - e \sin(A+B+C+D) &=0\label{eqn:sin}\end{align}\normalsize Upon setting $a = 1$ (we may set the scale factor of the pentagon as we like) and substituting the known angles and sides into Equations \ref{eqn:cos} and \ref{eqn:sin}, we find that \[c = \frac{1}{\sqrt{2}(\sqrt{3} - 1)}\] satisfies both equations.
Upon inspection, we see that this pentagon still does not match a known type. Thus, the pentagon with these side lengths and angles measures is a new type of pentagon (Type 15). This tile and a corresponding 3-block-tiling by this tile are shwon in Figure \ref{fig:type-15}.

\begin{figure}[H]\centering
\begin{subfigure}[H]{.35\textwidth} 
\centering
\includegraphics[scale=.6]{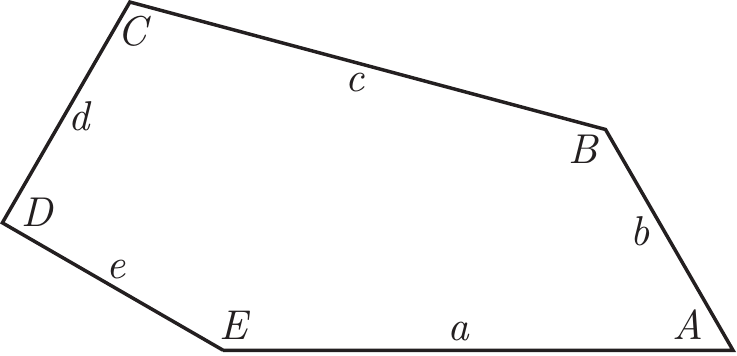}  
\caption{\scriptsize \begin{tabular}{ll} & \\
$A = 60^{\circ}$ & \hspace{.3in} $a = 1$ \\
$B = 135^{\circ}$ & \hspace{.3in} $b = 1/2$ \\
$C = 105^{\circ}$ & \hspace{.3in} $\displaystyle c = \frac{1}{\sqrt{2}(\sqrt{3}-1)}$\\ 
$D = 90^{\circ}$ & \hspace{.3in} $d = 1/2$\\ 
$E = 150^{\circ}$ &\hspace{.3in} $e = 1/2$\end{tabular}}\label{fig:15}
\end{subfigure}
\begin{subfigure}[H]{.6\textwidth}
\centering
\includegraphics[scale=.5]{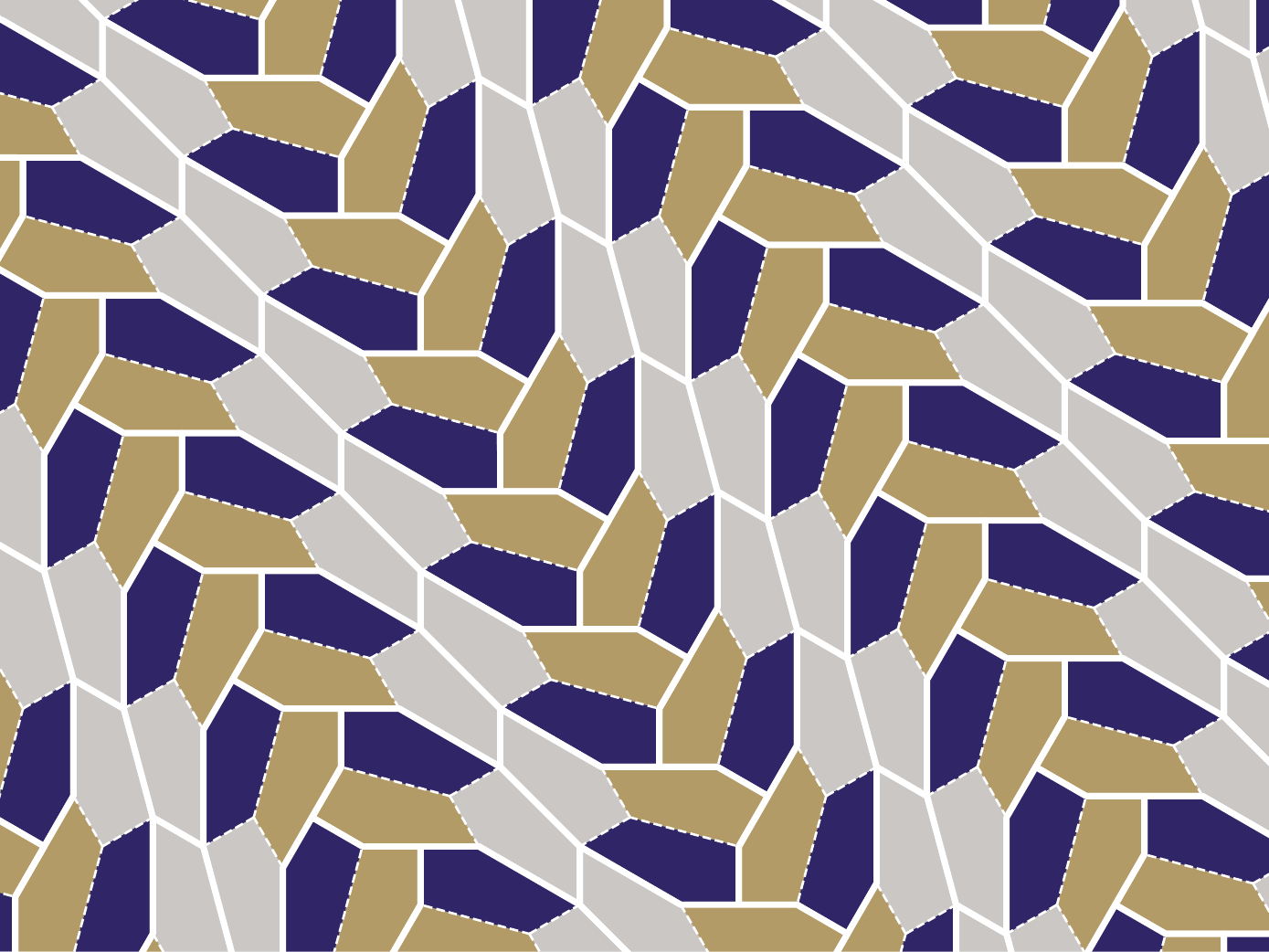}  
\caption{\scriptsize \mbox{}\\A 3-block transitive tiling by the Type 15 pentagon. The thick white lines outline the 3-block, and the colors of the tiles indicate the transitivity classes of pentagons.}\label{fig:tiling}
\end{subfigure}\caption{The Type 15 pentagon}\label{fig:type-15}
\end{figure}

\subsection{Untyped Solutions}\label{subsec:impossible}
Our computer code generated several sets of equations whose solutions did not automatically fall into Types 1-14 and also could not immediately be dismissed as impossible. Initially, these solutions were of extreme interest, for they might have represented new types of pentagons! However, it turned out that these solutions cannot be satisfied by any convex pentagon, or  geometric constraints will generate additional information so that such a pentagon must be of a known type. We call solutions such as these \emph{untyped.} Our computerized enumeration generated several untyped solutions. To keep this article to a reasonable length, we will not provide the details for how each of these untyped solutions was reconciled, but we mention that it required several separate nontrivial arguments to show that these untyped solutions are either impossible or can be categorized into the known 14 types. The following cases give a good representation of the types of arguments we gave for them all.

	\begin{enumerate}
	
	\item $A = 2\pi/3$, $B = 2\pi/3$, $C = \pi/2$, $D = 2\pi/3$, $E = \pi/2$, $b = 2a$, $e = d$.
	\item $C = \pi - A/2$, $D = 2\pi - 2B$, $E = B - A/2$, $a = b = d = e$
	\item $C= \pi - A$, $D=B$, $E=A$, $b=c$, and $d = e$ 
	\item $B=\pi - A/2$, $C=A/2+ \pi/2$, $D=\pi-A$, $E= \pi/2$, $b=2a+d$, $e=a+d$
	\item $B=\pi - A/2$, $C=A/2+ \pi/2$, $D=\pi-A$, $E= \pi/2$, $b+d=2a$, $e=a$
	\end{enumerate}
\subsubsection{Untyped Solution 1}
As is, this particular system looks very similar to the equations for a Type 3 pentagon, but it does not quite match. However, upon setting the scale factor of $a = 1$ (so $b = 2$), substituting into Equations \ref{eqn:cos} and \ref{eqn:sin}, and solving for $c$ and $d$, we obtain $c = 1$ and $d = \sqrt{3}$. With this new information, that $c = 1 = a$, we can positively type this set of equations as Type 3.

\subsubsection{Untyped Solution 2}
Using the relations in this system, we can reduce Equation \ref{eqn:sin} (with $A_1 = C$) to \begin{equation*} \sin(A/2) - \sin B + \sin(A+B) - \sin\left(\frac{A+4B}{2}\right) = 0.\end{equation*} Upon applying the sum-to-product identity for sine to the 1st and 4th terms and the 2nd and 3rd terms of this sum and factoring, we arrive at the equation \begin{equation*}2\cos\left(\frac{A+2B}{2}\right)[-\sin B + \sin(A/2)] =0.\end{equation*} Solving this equation for $B$ (with the restriction $0 < A,B < \pi$) gives $B = -A/2 + \pi/2$, $B = A/2$, or $B = -A/2 + \pi$. However, each of these solutions for $B$ is impossible. If $B = -A/2 + \pi/2$, then substitution into Equation \ref{eqn:cos} gives $c = 2 \cos(A/2) + 2\sin(A/2)$, and so in order that $c$ be positive we must have $A > \pi/2$.  But this implies $E = B - A/2 = -A + \pi/2 < 0$. If $B = A/2$, substitution into Equation \ref{eqn:cos} reveals that $c = 0$. Lastly, if $B = -A/2 + \pi$, then substitution into Equation \ref{eqn:cos} again implies $c = 0$. Thus, this system of equations cannot be realized by a convex pentagon.

\subsubsection{Untyped Solution 3}
For the untyped solution 3, Equation \ref{eqn:sin} along with the fact that $A+B+C+D+E = 3\pi$ gives \[(c-d)[\sin(A/2)+\sin(A)]=0.\] Note that for $0<A<\pi$ there are no solutions for \footnotesize\begin{equation*} 0=\sin(A/2)+\sin(A)=\sin(A/2)+2\sin(A/2)\cos(A/2)=\sin(A/2)[1+2\cos(A/2)].\end{equation*} \normalsize Hence $c=d$ so that $b=c=d=e$. Since $A+C = \pi$ and $b=d$, any pentagon satisfying these equations is Type 2.

\subsubsection{Untyped Solution 4}
For untyped solution 4, without loss of generality, assume $a=1$. Equation \ref{eqn:sin} gives $$-1+2 \sin(A) +c \sin(A/2)+d(-1-\cos(A)+\sin(A))= 0.$$ Note that $-1-\cos(A)+\sin(A)=0$ if and only if $A=\pi/2$. In that case, $A+E=\pi$ and the pentagon is a Type I. Otherwise, suppose $-1-\cos(A)+\sin(A) \neq 0$. Solving for $d$ we get \[d=\frac{1-c \sin(A/2)-2 \sin(A)}{-1-\cos(A)+\sin(A)}.\] Substitution into Equation \ref{eqn:cos} yields \[\frac{-1+c(-2 \cos(A/2+\sin(A/2))}{1+\cos(A)-\sin(A)}=0,\] from which we find that  \[c=\frac{1}{-2 \cos(A/2)+\sin(A/2)}.\] Using that $c>0$, we need $-2 \cos(A/2)+\sin(A/2)>0$ or $\tan(A/2)>2$. Since the tangent function is increasing on $(0,\pi/2)$, we get $\arctan(2) < A/2 < \pi/2$ or $2.21 \approx 2 \arctan(2) < A < \pi$. Observe that $-1-\cos(A)+\sin(A)>0$ for $2.21 \approx 2 \arctan(2) < A < \pi$. The requirement that $d>0$ gives $1-c \sin(A/2)-2 \sin(A)>0$ so that $c<\frac{1-2 \sin(A)}{\sin(A/2)}$. This inequality implies $$\frac{1}{-2 \cos(A/2)+\sin(A/2)} < \frac{1-2 \sin(A)}{\sin(A/2)}.$$ Since the denominators are positive we must have $$\sin(A/2) < (1-2 \sin(A))(-2 \cos(A/2)+\sin(A/2)).$$ This inequality is never satisfied for angles satisfying $2.21 \approx 2 \arctan(2) < A < \pi$.

\subsubsection{Untyped Solution 5}
For the 5th untyped solution, without loss of generality, assume $a=1$. Equation \ref{eqn:sin} gives $$-1-2\cos(A)+b(\cos(A)+\sin(A))+c \sin(A/2)= 0.$$ Note that $\cos(A)+\sin(A)=0$ if and only if $A=3\pi/4$. In this case, Equation \ref{eqn:sin} reduces to $-1+\sqrt{2}+c \cos(3\pi/8)=0$, yielding a negative value for $c$. Thus, we may suppose $\cos(A)+\sin(A) \neq 0$. Solving for $b$ in the Equation \ref{eqn:sin} gives $$b=\frac{1+2\cos(A)-c\sin(A/2)}{\cos(A)+\sin(A)}.$$ Substitution into Equation \ref{eqn:cos} gives $$\frac{2-2\sin(A)+c(\cos(A/2)+\sin(A/2))}{\cos(A)+\sin(A)}=0,$$ and solving for $c$ yields $$c=\frac{2(-1+\sin(A))}{\cos(A/2)+\sin(A/2)}.$$ From this we see that $c < 0$, and so this untyped solution is impossible.

\subsection{Summary of results obtained via computer for pentagons admitting 1-, 2-, and 3-block transitive tilings}\label{sec:result-summary}
\begin{table}[H]\centering
\begin{tabular}{|l|l|l|}\hline
Number of nodes &  Pentagon Types Found \\\hline\hline
$n = 5$ &  1, 2, 4, 5\\\hline
$n =6$ & 1, 2, 3 \\\hline\hline
$n = 10$ & 1, 2, 4, 5, 6, 7, 8, 9\\\hline
$n = 11$ & 1, 2, 4, 13\\\hline
$n = 12$ & 1, 2, 4, 11, 12\\\hline\hline
$n = 15$ &  1, 2, 5, 6, 7, 9 \\\hline
$n = 16$ & 1, 2, 3, 4, 5, 6, 15 \\\hline
$n = 17$ & 1, 2, 10\\\hline
$n = 18$ & 1, 2, 3, 10, 14 \\\hline
\end{tabular}\caption{Types of pentagons admitting $i$-block transitive tilings for $i = 1, 2$, and $3$.}\label{tab:1-3-block-results} \end{table}

\section{Future Work: $i\geq 4$}
As $i$ gets larger, the enumeration process outlined earlier grows rapidly in complexity. For relatively small $i$, the method outlined in this article is applicable with the aid of a cluster of computers. We are currently in the process of processing the pentagons that admit $i$-block tilings when $i \geq 4$, and will update this article with further results as we obtain them. The main challenge in extending this search is in efficiently understanding untyped solutions that arise. For a given untyped solution, one way to detect if the solution can be realized by a convex pentagon involves solving the system of equations given by Equations \ref{eqn:cos} and \ref{eqn:sin}. However, for many untyped solutions this system has 3 or more variables. Understanding the solution set for such a system is a challenge. Indeed, as seen in Section \ref{subsec:impossible}, there is no obvious way to automate the process of whether or not a given untyped solution can be realized by a convex pentagon, and if so, whether or not additional conditions will emerge that force such a pentagon to be among the known types.

\section{References}
\nocite{Rei}

\bibliography{refs2}{}
\bibliographystyle{apalike}






\pagebreak

\appendix

\section{Pentagons that admit tile-transtive tilings}\label{app:1-blocks}
Pentagons that admit tile-transitive tilings have already been classified \cite{HK}, but for the sake of illustrating our methods, we will offer our own verification here. 
\subsection{$n = 5$: Pentagons that admit edge-to-edge tile-transitive tilings}

Suppose a pentagon $P$ admits a tile-transtive tiling $\mathscr{T}$ in which each pentagon has exactly 5 vertices ($i = 1, n = 5$). From Table \ref{tab:Dio-solutions}, $\mathscr{T}$ must be of topological type $[3^3.4^2]$ or $[3^2.4.3.4]$, or $[3^4.6]$. These topological types corresponds to isohedral types IH21-IH29. For convenience we list the incidence symbols of isohedral types IH21-IH29 in Table \ref{tab:iso-types-21-29}. The goal is to examine each possible isohedral type for $\mathscr{T}$ to determine conditions on the angles and sides of $P$.

\begin{table}[H]\centering\scriptsize
\begin{tabular}{lllr}
Topological Type & Isohedral Type & Incidence Symbol & Edge Classes\\ \hline
$[3^4.6]$ & IH21 & $[a^+ b^+ c^+ d^+ e^+; e^+ c^+ b^+ d^+ a^+]$& $\alpha \beta \beta \gamma \alpha$\\ \hline
$[3^3.4^2]$ & IH22 & $[a^+ b^+ c^+ d^+ e^+; a^- e^+ d^- c^- b^+]$ & $\alpha \beta \gamma \gamma \beta$\\
 & IH23 & $[a^+ b^+ c^+ d^+ e^+; a^+ e^+ c^+ d^+ b^+]$& $\alpha \beta \gamma \delta \beta$\\
 & IH24 & $[a^+ b^+ c^+ d^+ e^+; a^- e^+ c^+ d^+ b^+]$&$\alpha \beta \gamma \delta \beta$\\
 & IH25 & $[a^+ b^+ c^+ d^+ e^+; a^+ e^+ d^- c^- b^+]$&$\alpha \beta \gamma \gamma \beta$\\
 & IH26 & $[ab^+ c^+ c^- b^-; ab^- c^+]$&$\alpha \beta \gamma \gamma \beta$\\\hline
$[3^2.4.3.4]$ & IH27 & $[a^+ b^+ c^+ d^+ e^+; a^+ d^- e^- b^- c^-]$&$\alpha\beta\gamma\beta\gamma$\\
 & IH28 & $[a^+ b^+ c^+ d^+ e^+; a^+ e^+ d^- c^- b^+]$&$\alpha\beta\beta\gamma\gamma$\\
 & IH29 & $[ab^+ c^+ c^- b^-; ac^+ b^+]$&$\alpha\beta\beta\beta\beta$\\
\end{tabular}\caption{Isohedral types IH21 - IH29 with edge transitivity classes}\label{tab:iso-types-21-29}\end{table}

For example, if $\mathscr{T}$ is of species type $\left<3^3.4^2\right>$, suppose $\mathscr{T}$ is type IH22. The first task is to determine the labelings of $P$ with $a^+b^+c^+d^+e^+$ that are compatible with the incidence symbol for IH22. For example, in Figure \ref{fig:1-5-IH22-good}, a pentagon in a tiling of topological type $[3^3.4^2]$ has been assigned a labeling consistent with isohedral type IH22.

\begin{figure}[H]\centering \includegraphics[scale=.8]{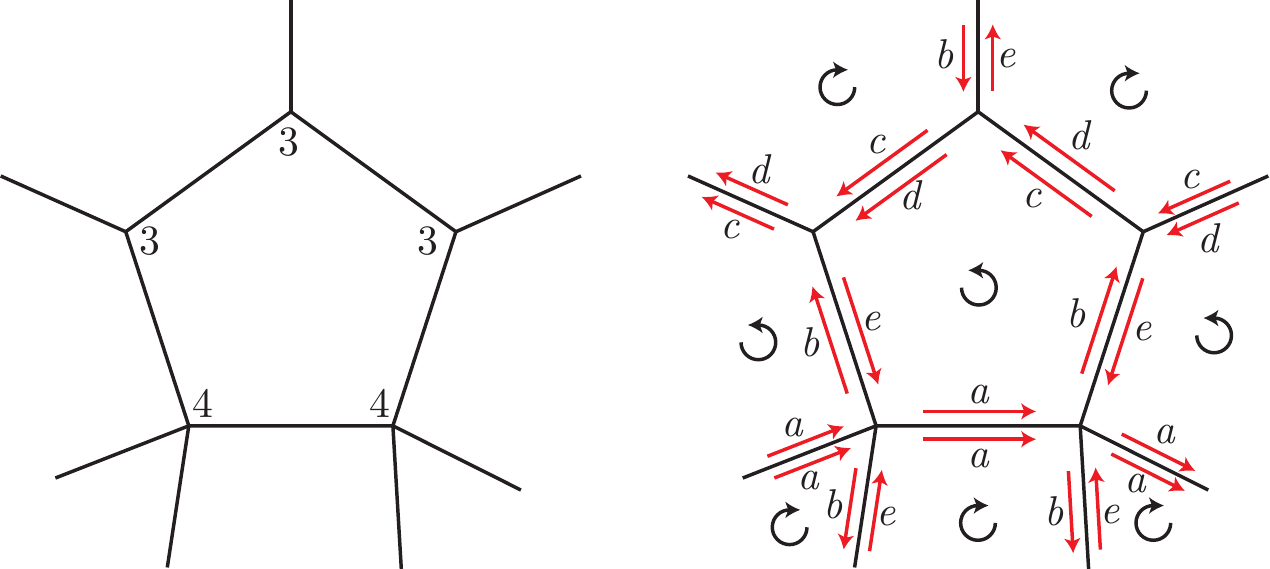} \caption{A compatible labeling of a pentagon of type IH22 (symbol $[a^+b^+c^+d^+e^+; a^-e^+d^-c^-b^+]$).} \label{fig:1-5-IH22-good}\end{figure} 
 
It is easily checked that the only labeling compatible with this symbol places the ``$a$'' between the two 4-valent vertices. 
Next, labels $A, B, C, D,$ and $E$ are assigned to the corner angles of $P$ and labels $a, b, c, d,$ and $e$ are assigned to the sides as in Figure \ref{fig:1-5-IH22-labeled}. \begin{figure}[H]\centering \includegraphics[scale=.8]{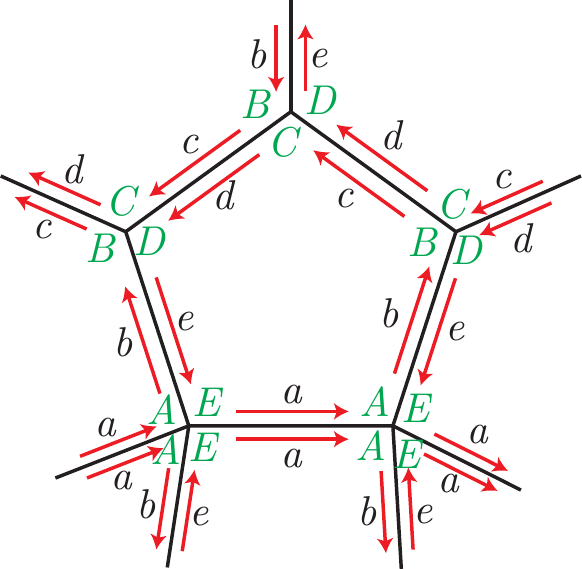} \caption{A pentagon of type IH22 with labeled angles. The side length labels correspond to the incidence labels in this case.} \label{fig:1-5-IH22-labeled}\end{figure} 
With this labeling, the required relationships among the angles and the sides may be read off, yielding \begin{align*}B + C + D & = 2\pi\\
A + E & = \pi\\
b & = e\\
c & = d\\\end{align*} In particular, because two consecutive angles of $P$ must be supplementary, we see that if $P$ admits an isohedral tiling of type IH22, then $P$ must be a Type 1 pentagon.

In a similar manner, it can be determined that the only compatible labeling for IH23-IH26 places the $a$ between the two 4-valent vertices as well. This in turn forces, $A+E = \pi$ for any pentagon admitting isohedral tilings of types IH23-IH26, and so any such pentagon is of Type 1. 

If $P$ admits tilings of isohedral types IH27, IH28, or IH29 the only compatible labeling requires that $a$ be placed between the two 3-valent vertices. This forces a unique labeling for pentagons of these isohedral types, as in Figure \ref{fig:IH27-29}. From these unique labelings, the equations corresponding to pentagons of types IH27, IH28, and IH29 are determined (Table \ref{tab:IH27-29-eqns}), from which we see it is seen that any pentagon admitting types IH27, IH28, or IH29 are pentagons of Types 2, 4, or 4 (respectively).

\begin{figure}[H]
\centering
\begin{subfigure}[H]{.3\textwidth} 
\centering
\includegraphics[scale=.6]{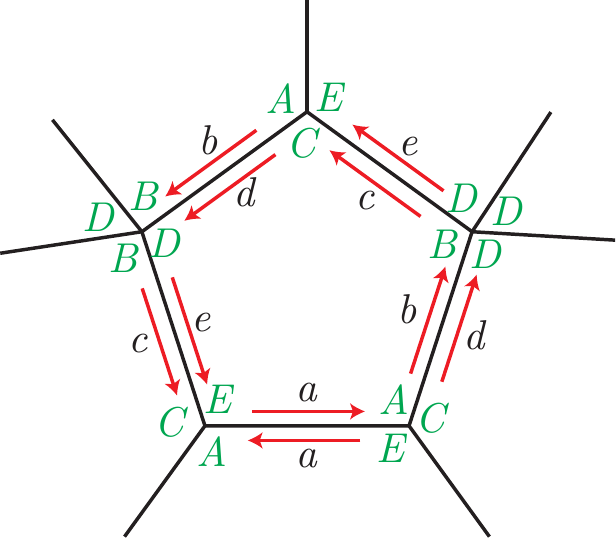}
\caption*{IH27}
\end{subfigure}
\begin{subfigure}[H]{.3\textwidth}
\centering
\includegraphics[scale=.6]{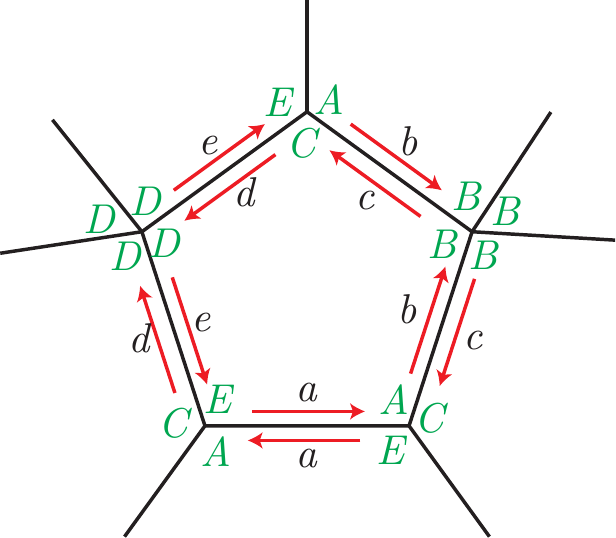}
\caption*{IH28}
\end{subfigure}
\begin{subfigure}[H]{.3\textwidth}
\centering
\includegraphics[scale=.6]{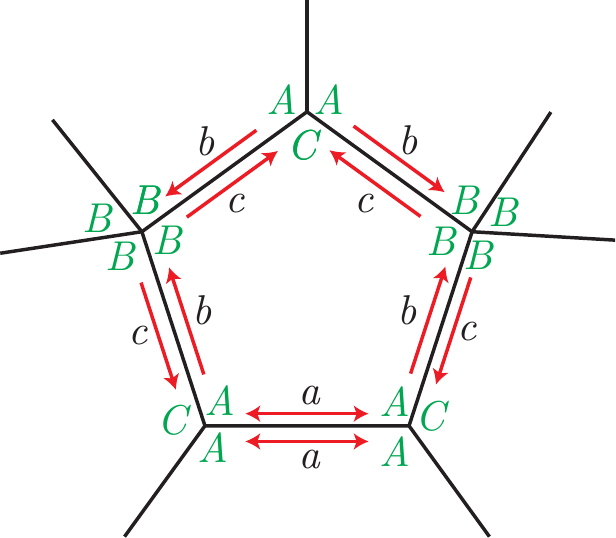}
\caption*{IH29}
\end{subfigure}
\caption{IH27, IH28, and IH29 pentagons with $n = 5$.}\label{fig:IH27-29}
\end{figure}  

\begin{table}[H]\centering
\begin{tabular}{lll}
IH27 & IH28 & IH29\\\hline
$B+D = \pi$ & $B = D = \pi/2$ & $B = \pi/2$\\
$b = d$ &$b = c$ & $2A + C = \pi$\\
$c = e$ & $d = e$ & $b = c$\\\end{tabular}\caption{Angle/side equations for pentagons with $n = 5$ of types IH27, IH28, and IH29}\label{tab:IH27-29-eqns}\end{table}

IH21 is the only isohedral type for topological type $[3^4.6]$. There are only two viable labelings of an IH21 pentagon corresponding to the incidence symbol for IH21 in Table \ref{tab:iso-types-21-29}. These labelings are seen in Figure \ref{fig:1-5-IH21}, and the required equations relating angles and sides are given in Table \ref{tab:1-5-IH21-eqns}. Both IH21 pentagons with $n = 5$ must be Type 5 if they are to tile the plane.

 \begin{figure}[H]
\centering
\begin{subfigure}[H]{.4\textwidth} 
\centering
\includegraphics[scale=.75]{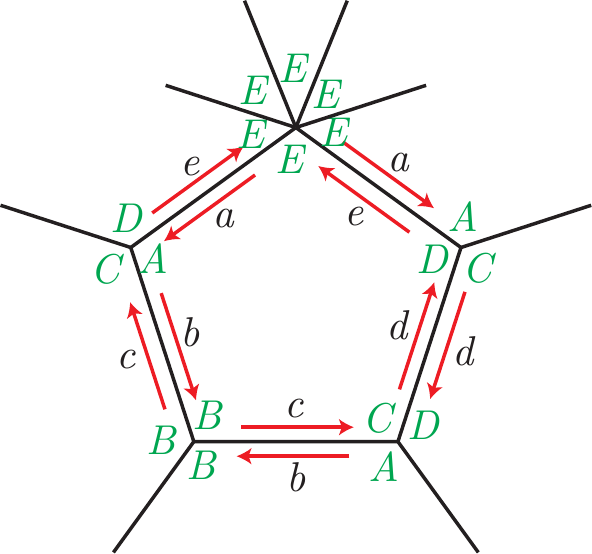}
\caption*{IH21(1)}
\end{subfigure}
\begin{subfigure}[H]{.4\textwidth}
\centering
\includegraphics[scale=.75]{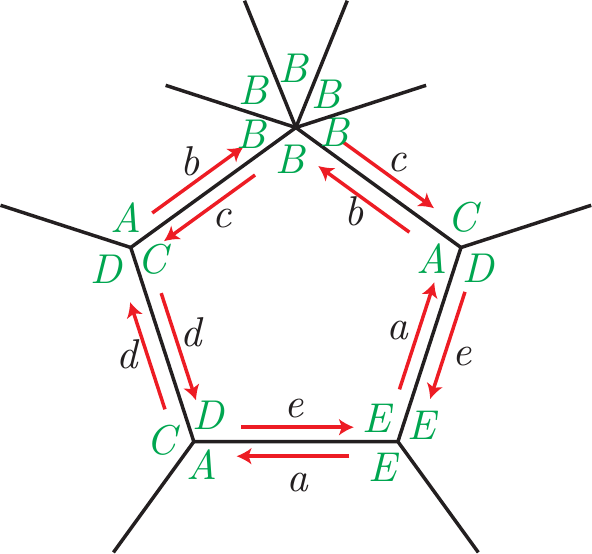}
\caption*{IH21(2)}
\end{subfigure}
\caption{IH21 pentagons with $n = 5$.}\label{fig:1-5-IH21}
\end{figure}   

\begin{table}[H]\centering
\begin{tabular}{ll}
IH21(1) & IH21(2) \\\hline
$E = \pi/3$ & $B = \pi/3$ \\
$B = 2\pi/3$ &$E = 2\pi/3$\\
$a = e, b = c$ & $a = e, b = c$\\\end{tabular}\caption{Angle/side equations for pentagons with $n = 5$ of types IH21}\label{tab:1-5-IH21-eqns}\end{table}

\noindent Other labelings of IH21 pentagons with $n = 5$ yield impossible relationships among the angles of the pentagon. For example, in Figure \ref{fig:1-5-IH21-3-bad}, the labeling requires $A+C+D = 2(A+C+D)$.

\begin{figure}[H]\centering \includegraphics[scale=.75]{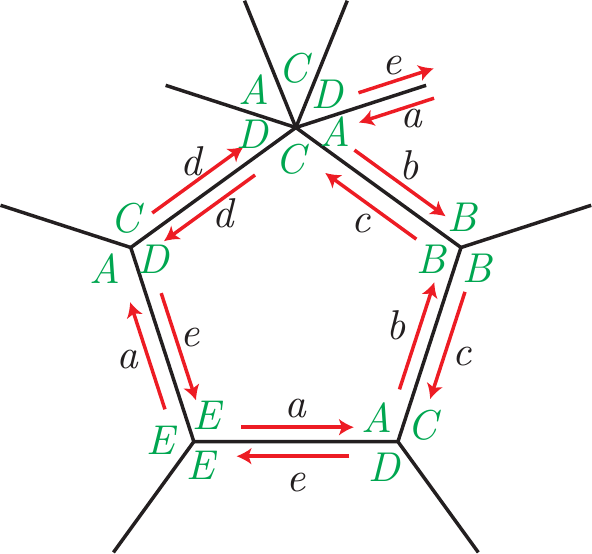}\caption{An impossible labeling of an IH21 pentagon with $n = 5$.} \label{fig:1-5-IH21-3-bad}\end{figure} 

The results for isohedral pentagons with $n = 5$ are summarized in Table \ref{table:nonlin}.

\subsection{$n = 6$: Pentagons that admit non-edge-to-edge tile-transitive tilings}

For $i=1$ and $n=6$, the only possible topological type is $[3^6]$. In this case, each pentagon of $\mathscr{T}$ has exactly one flat note appearing between two of the corners of the pentagon. Many isohedral types under topological type $[3^6]$ are impossible for such a pentagon. If a pentagon $P$ with $n = 6$ is labeled according to a $[3^6]$ isohedral type, consider an edge label $x$ from the incidence symbol that is adjacent to this flat node. In isohedral types IH8-IH11, IH18, and IH20, we see that each label must appear at least twice in $P$ and in nonadjacent locations. For these types, another side of $P$ that is not adjacent to the flat node must be labeled with $x$. This forces one of the corners of $P$ to have angle measure $\pi$, which cannot be (see Figure \ref{fig:1-6-nontilers-1}. In a similar manner, a label $x$ in the label that is adjacent to a flat node cannot be unsigned (see Figure \ref{fig:1-6-nontilers-2}. This observation in combination with the previous observation eliminates IH12 and IH13. For isohedral types IH17 and IH19, if in labeling $P$ we attempt to avoid labeling inconsistencies, we find that the symbols adjacent to the flat node must be of the form $x^{+}x^{-}$ or $x^{-}x^{+}$. However, in these two isohedral types, the edges adjacent to a corner of $P$ would necessarily be labeled $x^{+}x^{-}$ or $x^{-}x^{+}$, forcing that corner to be flat.

\begin{figure}[H]
\centering
\begin{subfigure}[H]{.4\textwidth} 
\centering
\includegraphics[scale=.75]{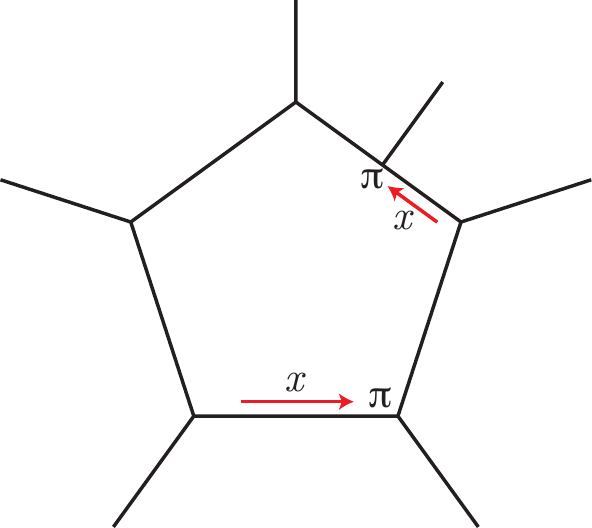}
\caption{}\label{fig:1-6-nontilers-1}
\end{subfigure}
\begin{subfigure}[H]{.4\textwidth}
\centering
\includegraphics[scale=.75]{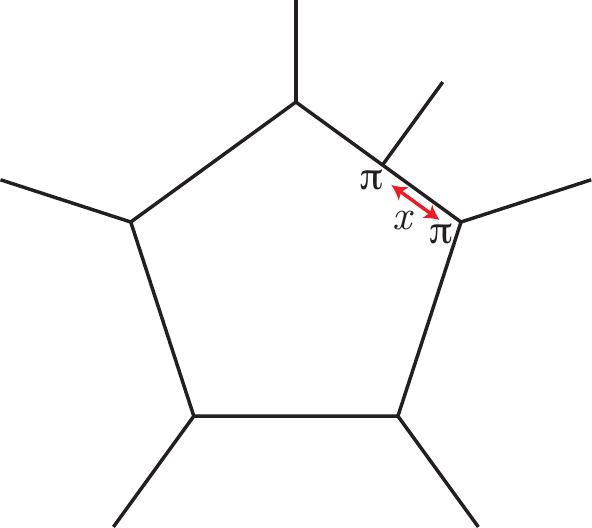}
\caption{}\label{fig:1-6-nontilers-2}
\end{subfigure}
\caption{Symbols that force flat corners in $P$.}\label{fig:1-6-nontilers}
\end{figure}   

After eliminating those isohedral types that are force $P$ to have a flat corner, types IH1-IH7 and IH14-IH16 remain to be checked. Any 6-node pentagon of isohedral type IH1-IH7 can be labeled in 6 ways (each labeling corresponding to the choice of symbols surrounding the flat node). Analyzing each possible labeling is a matter of routine, and from among these 42 labelings, 5 types of pentagons are found.
\begin{itemize}
\item Type 1 pentagons
\item Type 2 pentagons
\item Type 3 pentagons
\item Obviously impossible pentagons
\item Non-obviously impossible pentagons
\end{itemize}

\noindent Examples of labelings leading to these 5 outcomes will be presented next.

In Figure \ref{fig:1-6-IH2-type-I}, we see a labeling of a pentagon $P$ which forces two adjacent angles of $P$ to be supplementary, and so such a pentagon is of Type 1. Indeed any IH2 labeling of a 6-node pentagon yields a Type 1 pentagon. In Figure \ref{fig:1-6-IH3-type-II}, a 6-node pentagon has been given an IH3 labeling, and it is quickly determined such a pentagon is of Type 2. In Figure \ref{fig:1-6-IH7-type-III}, a 6-node pentagon is labeled as an IH7 tile. This labeling gives a Type 3 pentagon.

\begin{figure}[H]
\centering
\begin{subfigure}[H]{.3\textwidth} 
\centering
\includegraphics[scale=.6]{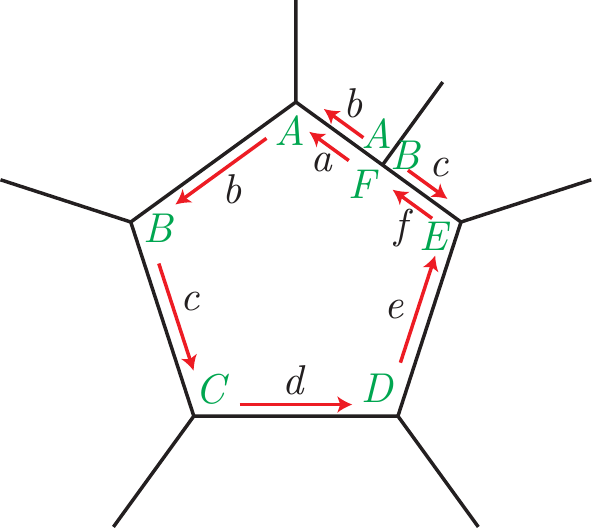}
\caption{}\label{fig:1-6-IH2-type-I}
\end{subfigure}
\begin{subfigure}[H]{.3\textwidth}
\centering
\includegraphics[scale=.6]{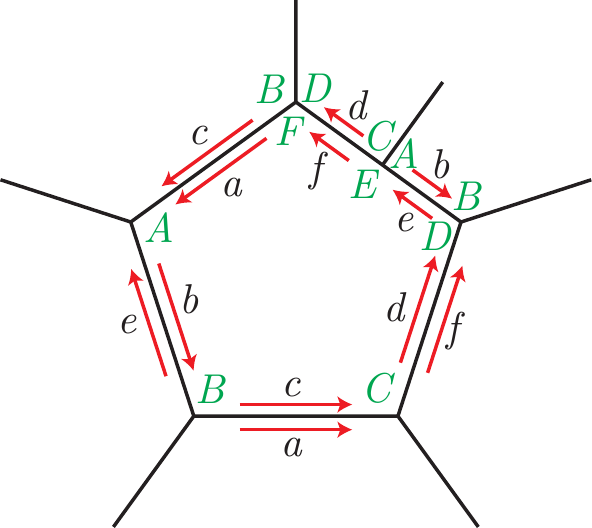}
\caption{}\label{fig:1-6-IH3-type-II}
\end{subfigure}
\begin{subfigure}[H]{.3\textwidth}
\centering
\includegraphics[scale=.6]{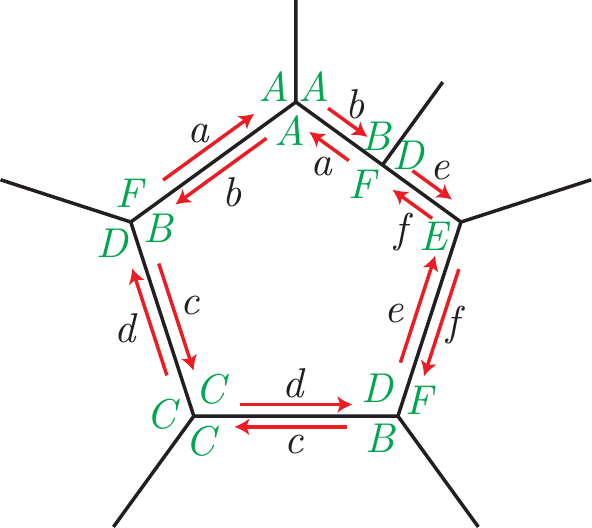}
\caption{}\label{fig:1-6-IH7-type-III}
\end{subfigure}
\caption{$[3^6]$ 6-node pentagons of Types 1, 2 ,and 3}\label{fig:1-6}
\end{figure}  

Most of the IH1-IH7 labelings of 6-node pentagons are easily categorized into the known 14 types, but two kinds of labelings arise that cannot be realized by an actual convex pentagon. We will refer to such labelings as \emph{impossible}. The first impossible labeling, which appears in only three of the IH7 labelings, is impossible since three flat angles cannot surround a vertex (see Figure \ref{fig:1-6-IH7-bad}). The second type of impossible labeling is not obviously impossible. This labeling appears in equivalent forms in all six IH1 labelings and in two of the IH3 labelings. Consider the labeling of the 6-node pentagon of type IH1 in Figure \ref{fig:1-6-IH1-bad}. This labeling implies a geometrically impossible pentagon: a routine calculation reveals that the distance from the interior vertex labeled $B$ to the interior vertex labeled $D$ must be greater than $c + d$. Indeed, if the edge $\overline{EA}$ is placed on a horizontal with $E$ at the origin, then $B = (c + d + b \cos(\pi - A),b \sin(\pi - A))$ and $D = (b \cos E, b \sin E)$. Then \begin{align*}
|BD|^2 = & [c + d + b\cos(\pi - A) -  b\cos E)]^2 + [b \sin(\pi - A) - b \sin E]^2\\
= & (c+d)^2 + 2b(c+d)[\cos(\pi - A) -  \cos E] + b^2[\cos(\pi - A) - \cos E]^2\\
 & + b^2[\sin(\pi - A) - \sin E]^2 \\
\geq & (c+d)^2 + 2b(c+d)[\cos(\pi - A) -  \cos E] \\
> & (c+d)^2\end{align*} Since $A+C+E = 2\pi$ and all interior angles of a convex pentagon are less than $\pi$, then $A+E > \pi$, so $\pi > A > \pi - E  > 0$ and $\cos$ is decreasing on the interval $[0,\pi]$, which justifies the final inequality.

\begin{figure}[H]
\centering
\begin{subfigure}[H]{.4\textwidth} 
\centering
\includegraphics[scale=.75]{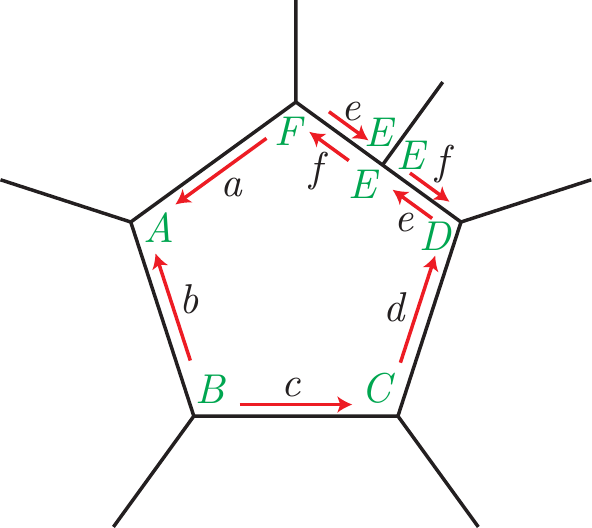}
\caption{$E = \pi$}\label{fig:1-6-IH7-bad}
\end{subfigure}
\begin{subfigure}[H]{.4\textwidth}
\centering
\includegraphics[scale=.75]{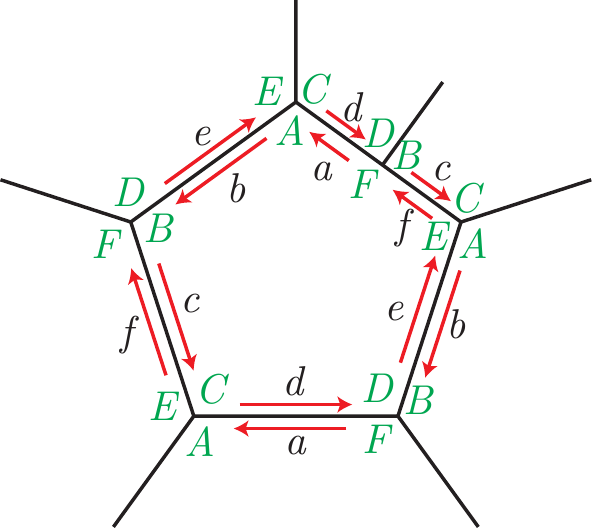}
\caption{$F = \pi$}\label{fig:1-6-IH1-bad}
\end{subfigure}
\caption{Two inconsistent labelings of $[3^6]$ 6-node pentagons from types IH1-IH7}\label{fig:1-6-nontilers}
\end{figure}

Next consider the IH14-IH16 labelings of 6-node pentagons. These three isohedral types are similar in that the incidence symbols require, for the same reasons previously discussed pertaining to labeling of edges adjacent to the flat node, that the edges adjacent the the flat node must be marked $a^{-}a^{+}$ or $c^{+}c^{-}$, so there are only two viable labelings for each of these three isohedral types. The two viable labelings for IH14 produce pentagons like the one of Figure \ref{fig:1-6-IH1-bad}, so there are no possible tilings by 6-node pentagons of isohedral type IH14. The two viable IH15 labelings are shown in Figure \ref{fig:1-6-IH15}, and the resulting pentagons are of Type 1. Isohedral type IH16 yields the two labelings of Figure \ref{fig:1-6-IH16}. Figure \ref{fig:1-6-IH16-type-III} gives a Type 3 pentagon, and Figure \ref{fig:1-6-IH16-bad} is impossible.

\begin{figure}[H]
\centering
\begin{subfigure}[H]{.4\textwidth} 
\centering
\includegraphics[scale=.75]{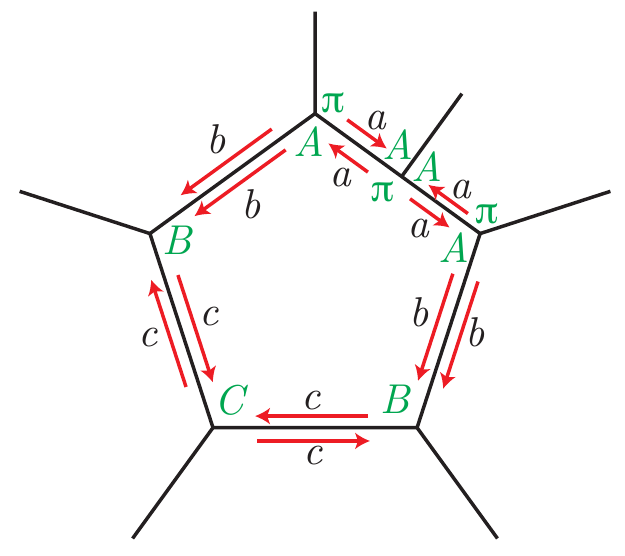}
\caption{}
\end{subfigure}
\begin{subfigure}[H]{.4\textwidth}
\centering
\includegraphics[scale=.75]{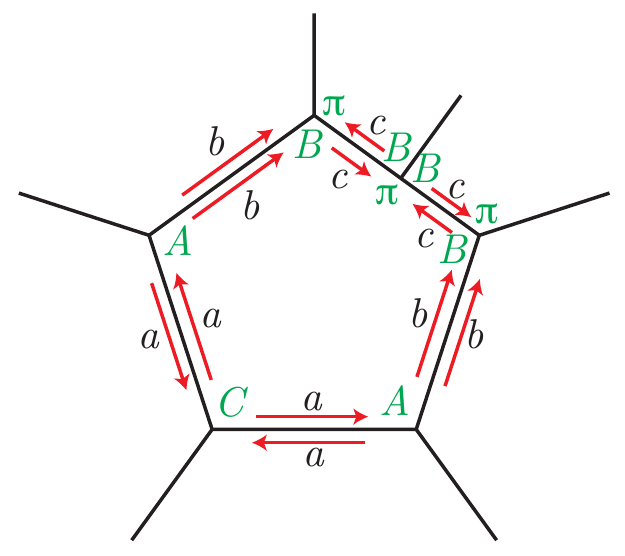}
\caption{}
\end{subfigure}
\caption{6-node IH15 pentagons}\label{fig:1-6-IH15}
\end{figure}   

\begin{figure}[H]
\centering
\begin{subfigure}[H]{.4\textwidth} 
\centering
\includegraphics[scale=.75]{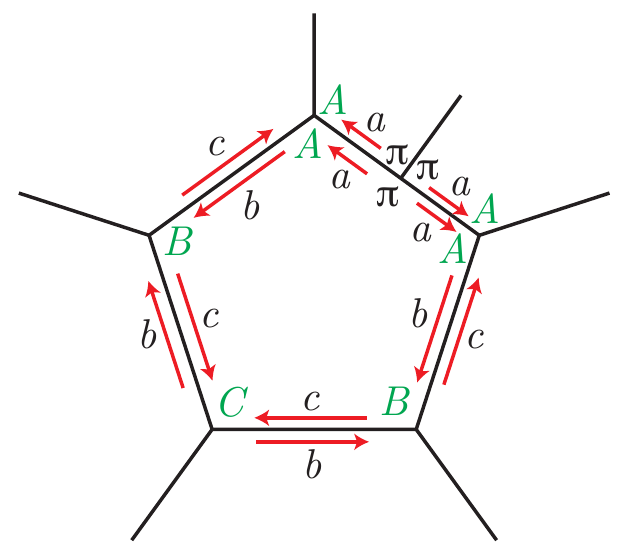}
\caption{}\label{fig:1-6-IH16-bad}
\end{subfigure}
\begin{subfigure}[H]{.4\textwidth}
\centering
\includegraphics[scale=.75]{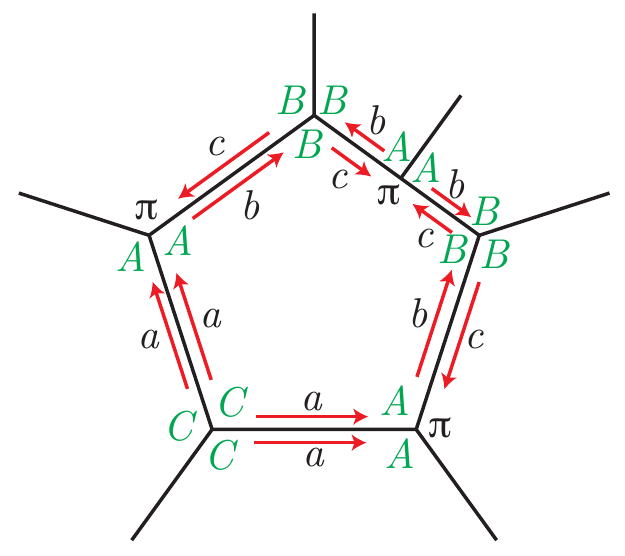}
\caption{}\label{fig:1-6-IH16-type-III}
\end{subfigure}
\caption{6-node IH16 pentagons}\label{fig:1-6-IH16}
\end{figure} 

\begin{table}[H]
\centering\scriptsize
\begin{tabular}{|c c c c c c|} 
\hline
$i$ & $n$ & Vertex Valences & Topolgical Type & Isohedral Type & Possible Pentagon Type(s)  \\
\hline
1 & 5 & $\left<3.3.3.3.6\right>$ & $[3^4.6]$ & IH21 & 5\\ \hline
 &  & $\left<3.3.3.4.4\right>$ & $[3^3.4^2]$ & IH22 & 1\\ 
 & &  & $[3^3.4^2]$ & IH23 & 1\\ 
&  &  & $[3^3.4^2]$ & IH24 & 1\\ 
&  &  & $[3^3.4^2]$ & IH25 &  1 \\ 
&  &  & $[3^3.4^2]$ & IH26 & 1 \\ 
&  &  & $[3^2.4.3.4]$ & IH27 & 2\\
&  &  & $[3^2.4.3.4]$ & IH28 & 4\\
&  &  & $[3^2.4.3.4]$ & IH29 & 4\\ \hline
1& 6 & $\left<3.3.3.3.3.3\right>$ & $[3^6]$ & IH1 & -\\ 
 &  &  & $[3^6]$ & IH2 & 1\\
 &  &  & $[3^6]$ & IH3 & 2\\
 &  &  & $[3^6]$ & IH4 & 1\\
 &  &  & $[3^6]$ & IH5 & 1\\
 &  &  & $[3^6]$ & IH6 & 1,2\\
 &  &  & $[3^6]$ & IH7 & 3\\
 &  &  & $[3^6]$ & IH8-14 & - \\
 &  &  & $[3^6]$ & IH15 & 1\\
 &  &  & $[3^6]$ & IH16 & 3\\
 &  &  & $[3^6]$ & IH17-20 & -
\\ [1ex] 
\hline 
\end{tabular}
\caption{Types of Isohedral Pentagons}\label{table:nonlin} 
\end{table}

In summary, we see from Table \ref{table:nonlin} that any pentagon that admits a tile transitive tiling of the plane must be of the known types 1 - 5, confirming the result in \cite{HK}.

\end{document}